\documentclass{amsart}
\usepackage{amsmath,amsfonts,amsthm,amssymb}
\usepackage[alphabetic]{amsrefs}
\usepackage{hyperref}
\usepackage{enumerate}
\usepackage{color}

\newtheorem{lem}{Lemma}[section]
\newtheorem{tw}[lem]{Theorem}

\newtheorem{cor}[lem]{Corollary}
\newtheorem{prop}[lem]{Proposition}

\newtheorem{fact}[lem]{Fact}
\newcommand {\rem}{\noindent {\bf Remark. }}
\newcommand {\rems}{\noindent {\bf Remarks. }}

\newcommand {\mr}{\mathrm}
\newcommand {\lk}{\left\{}
\newcommand {\rk}{\right\}}

\newtheorem*{twA}{Theorem A}
\newtheorem*{twB}{Corollary B}
\newtheorem*{twC}{Theorem C}
\newtheorem*{twD}{Theorem D}
\newtheorem*{twE}{Theorem E}
\newtheorem*{twF}{Theorem F}
\newtheorem*{twG}{Theorem G}
\newtheorem*{twH}{Theorem H}

\theoremstyle{definition}
\newtheorem{exs}[lem]{Examples}
\newtheorem{de}[lem]{Definition}

\newcommand {\dam}{\marginpar{\tiny }}

\begin{document}

\title{On asymptotically hereditarily aspherical groups}
\author{Damian Osajda}
\address{Universit\"at Wien, Fakult\"at f\"ur Mathematik\\
Oskar-Morgenstern-Platz 1, 1090 Wien, Austria}
\address{Instytut Matematyczny, Polska Akademia Nauk\\
\'Sniadeckich 8, 00-656 War\-sza\-wa, Poland}
\email{dosaj@math.uni.wroc.pl}
\author{Jacek \'Swi\c atkowski}
\address{Instytut Matematyczny,
Uniwersytet Wroc\l awski\\
pl.\ Grunwaldzki 2/4,
50--384 Wroc\-{\l}aw, Poland}
\email{swiatkow@math.uni.wroc.pl}

\thanks{Both authors were partially supported by MNiSW
grant N201 541738, and by Narodowe Centrum Nauki, decision no DEC-2012/06/A/ST1/00259.}
\date{\today}

\begin{abstract}
We undertake a systematic study of asymptotically hereditarily aspherical (AHA) groups -- the class of groups introduced by Tadeusz Januszkiewicz and the second author as a tool for exhibiting exotic properties of systolic groups. We provide many new examples of AHA groups, also in high dimensions. We relate the AHA property with the topology at infinity of a group, and deduce in this way some new properties of (weakly) systolic groups. We also exhibit an interesting property of boundaries at infinity for a few classes of AHA groups.
\end{abstract}
\subjclass[2010]{20F69 (Primary), 20F65 (Secondary)}
\keywords{asymptotic hereditary asphericity, (weakly) systolic group, coarse invariant}
\maketitle


\section{Introduction}
\label{intro}

{\it Asymptotic hereditary asphericity} (shortly AHA) is a coarse
property of metric spaces which reflects, at the asymptotic level, the fact
that every subspace of the space is aspherical (see \dam Definition \ref{daha} for
a precise statement). This property emerged in the work
of Tadeusz Januszkiewicz and the second author \cite{JS1}, where it was used
to show that certain high dimensional word hyperbolic groups
(namely systolic groups) do not contain
high dimensional arithmetic subgroups. Systolic groups were defined
in \cite{JS1} (independently in \cite{Hag}) and are related to the notion of simplicial nonpositive curvature. In \cite{JS2} it is proved that systolic groups are AHA.

In the present paper we provide some evidence
that the property AHA deserves further attention in the geometric study of groups.
First of all, we notice that there are examples of AHA groups other than
systolic ones. For example,
as it is shown in \cite{Z}, all groups
with asymptotic dimension $1$ are AHA. This includes e.g.\ the lamplighter groups
which are not finitely presented and hence not systolic.
In this paper, in Section \ref{twocpl},
we show that many $2$--dimensional groups are AHA,
revealing an interesting relationship between AHA and the celebrated
Whitehead's asphericity conjecture.
In particular, the Baumslag-Solitar groups (which are not
systolic, since they are not automatic) are AHA.
(As a side remark we pose the following
problem, which may be viewed as a
coarse variant of Whitehead's conjecture for groups: \emph{is every $2$--dimensional
asymptotically aspherical group AHA?} See the end of Section~\ref{twocpl} for a more
precise statement of this question, and for some comments.)

Perhaps the most interesting aspect of the AHA property is that it
is satisfied by certain groups of dimension above $2$.
In fact, it was already shown in \cite{JS2} that there are the AHA groups,
namely systolic ones, which have arbitrary asymptotic or cohomological
dimension. In this paper we show that groups related to spaces satisfying a
certain more general condition than systolicity
are also AHA (see Theorem \ref{systaha}), and we exhibit examples
of arbitrary dimension among them.
More precisely, in \cite{O-sdn} and \cite{O-chcg} the first author
has isolated a condition that unifies simplicial nonpositive
and cubical negative curvature, under the name {\it weak systolicity}.
We show that groups from some general subclass of weakly systolic groups
are all AHA, and we construct examples of such groups, in arbitrary
cohomological dimension (see Subsection~\ref{chcg}).

Asymptotic hereditary asphericity is obviously a rather
unexpected phenomenon among groups of high (cohomological) dimension.
Nevertheless, it might be true that ``generic'' groups of dimension above 2
are AHA, despite our common impression that typical high dimensional
phenomena are similar to that of high dimensional manifolds.
As we show in this paper (see Lemma~\ref{prescom} and Example~\ref{4.4}(4)),
generic (random) groups resulting from Gromov's density model are AHA.
However, these groups are known to have cohomological dimension $2$.
Procedures producing generic high dimensional
groups seem to be not known.

Another motivation for our interest in the AHA property comes from
a (partially verified) expectation
that various exotic phenomena among high dimensional systolic groups,
as discovered e.g.\ in \cites{JS2,O-ciscg,O-ib7scg,Sw-propi},
get proper perspective in
a more general setting of AHA groups.  In this direction,
we show in particular that all finitely
presented AHA groups are aspherical at infinity (Corollary~\ref{FpiAHA}), and
that Gromov boundaries of word hyperbolic AHA groups contain no $2$--disk
(Theorem~\ref{d1}(1)).
The respective results for systolic and $7$--systolic groups were earlier
established in \cites{O-ciscg,Sw-propi}.
Questions concerning other such properties remain open. For example,
we do not know whether the boundary of any AHA group is hereditarily
aspherical (we know this only for Gromov boundaries of some hyperbolic
systolic groups, see \cites{O-ib7scg,O-sdn}).

By appealing to the AHA property and its general consequences,
we have established in this paper a few
new facts about systolic groups. For example,
we show that the fundamental group of a closed manifold covered by $\mathbb R^n$,
for $n\geqslant 3$, cannot be a subgroup of a systolic group (Corollary~\ref{nmans}).
We also show that
the systolic boundary (as defined in \cite{OP})
of any systolic group contains no $2$--disk (Theorem~\ref{d2}).

For a final motivating remark, note that existence of high dimensional AHA groups,
together with their properties established in this paper,
yield some interesting phenomena.
Let $G$ be an AHA group of high cohomological dimension.
We may additionally assume that $G$ is word hyperbolic.
Then one of the following two cases occurs:
\begin{enumerate}
\item no subgroup of $G$ is isomorphic to the fundamental group of a closed aspherical manifold
of dimension at least $3$;
\item there exists a subgroup $H<G$ which is isomorphic to the fundamental group of a closed
aspherical manifold $M$ of dimension $3$ or more.
\end{enumerate}
Case (1) would provide first known hyperbolic groups of high dimension and with this property.
In case (2) the subgroup $H=\pi_1(M)$ is itself AHA and, in particular, $\pi_i^{\infty}(H)=0$, for $i\geqslant 2$ (see Corollary~\ref{FpiAHA} below).
At the moment no manifold $M$ with such properties is known.
It seems more likely for us that case (1) is true.

\medskip
We now discuss the organization of the paper and give more precise statements of
our main results. After various preliminaries gathered in Section~\ref{prel}, we
prove in Section~\ref{scrit} the following useful criterion sufficient for a metric space or a group to be AHA.
For a subset $A$ in
a metric space, $N_D(A)$ denotes the set of all points at distance $\leqslant D$
from $A$ -- the $D$-{\it neighbourhood} of $A$. For a metric space $X$ and a real number $r>0$, the property 
{fill-rad$(r)<\infty$ means that there is $t>0$ such that each loop of length at most $r$ in $X$ can be
null-homotoped in its $t$-neighborhood in $X$ (see Section~\ref{scrit} for details).} An isometric action of a group $G$ is {\it proper}
if any metric ball intersects at most finitely many of its $G$-translates.

\begin{twA}[see Corollary \ref{ccrit} in the text]
\label{twa}
Let $G$ be a group acting {properly by isometries} on a simply connected geodesic
metric space $X$ {with fill-rad$(r)<\infty$, for each $r>0$}. Suppose that to each subset
$A\subseteq X$ there is assigned a subspace $Y_A$ of $X$ such that:
\begin{itemize}
\item {\it $A\subseteq Y_A$ and $Y_A\subseteq N_D(A)$ for some universal
$D$ independent of $A$;}
\item {\it $Y_A$ is aspherical;}
\item {\it if $A_1\subseteq A_2$ then $Y_{A_1}\subseteq Y_{A_2}$.}
\end{itemize}
\noindent
Then $X$, and hence also $G$, is AHA.
\end{twA}

Using this criterion, we obtain in Section~\ref{twocpl} the following.

\begin{twB}[see Lemma~\ref{asf2} and Examples~\ref{4.4} in the text]
Let $G$ be a group acting geometrically {(i.e.\ by isometries, properly and cocompactly)}
on a simply connected $2$--dimensional geodesic metric cell complex, whose every
subcomplex is aspherical. Then $G$ is AHA. In particular,
groups from the following classes are AHA:
\begin{enumerate}
\item finitely presented small cancellation groups,
\item groups acting {geometrically} on $2$--dimensional CAT(0) complexes,
\item random groups of Gromov,
\item one relator groups,
\item knot groups,
\item fundamental groups of compact $2$--complexes satisfying Whitehead's
conjecture.
\end{enumerate}
\end{twB}

Note that (1) and (4) above give answers to the questions posed by
A.\ Dranishnikov and D.\ Osin in \cite[Problems 3.15(b) and 3.16]{Dra}.
Moreover, (1) generalizes to 
graphical small cancellation groups (see Example~\ref{4.4}(3))
and to
infinitely presented
small cancellation groups (see Lemma~\ref{prescom'} and Remarks afterwards). For such groups
with infinite asymptotic dimension (for example, the monsters from \cite{O-sc}) the AHA property is a very useful tool
(see Remarks at the end of Section~\ref{FAHA}).

Using again the criterion from Theorem A,
we prove in Section~\ref{weak} that groups from a certain subclass of so called
\emph{weakly systolic} groups (introduced by the first author in \cite{O-sdn})
are AHA. This class contains in particular all systolic groups, so this
generalizes the earlier mentioned result from \cite{JS1}
that all systolic groups are AHA.
Based on a recent elementary construction of highly dimensional weakly systolic groups provided
in \cite{O-chcg}, we give a new (i.e.\ different than the systolic ones provided e.g.\ in \cites{JS0,JS1})
construction of AHA groups in arbitrary cohomological dimension -- see Subsection \ref{chcg}.

In Section \ref{FAHA} we show the following.

\begin{twC}[see Theorem~\ref{AHAF} in the text]
Any finitely presented AHA group has type $F_\infty$, i.e.\ it has
a classifying CW complex with finitely many cells in each dimension.
\end{twC}

In Sections \ref{CIAHA} and \ref{AsAHA} we deal with connectedness at infinity of AHA groups.
We start Section \ref{CIAHA} by clarifying the definition of the fact
that $\pi_i^\infty(G)=0$ for a group $G$ (see Theorem \ref{qipi} and Definitions \ref{nasf} and \ref{gasf}). Next, we prove the following two results,
and conclude the section with their few corollaries of independent interest.

\begin{twD}[see Theorem \ref{piAHA} and Corollary \ref{FpiAHA} in the text]
Let $G$ be an AHA group acting geometrically on an $n$--connected
metric space. Then $\pi^\infty_i(G)=0$ for every $2\leqslant i\leqslant n$.
In particular, a finitely presented AHA group $G$ satisfies
$\pi^\infty_i(G)=0$ for all $i\geqslant 2$.
\end{twD}

\begin{twE}[see Theorem \ref{npi1} in the text]
Let $G$ be a finitely presented one-ended AHA group with
$vcd(G)<\infty$. Then $G$ is not simply connected at infinity,
i.e.\ it is not true that $\pi_1^\infty(G)=0$.
\end{twE}

In Section \ref{AsAHA} we introduce a new notion of
{\it asymptotic connectedness at infinity},
$\hbox{as--}\pi_i^\infty(X)=0$, and show that it is invariant under coarse
equivalence of metric spaces. The advantage (over classical connectedness
at infinity, which is more topological in flavor) is that this notion can
be studied for groups which are not necessarily finitely presented.
Accordingly, we get the following.

\begin{twF}[see Proposition \ref{AHAaspi0} in the text]
If $G$ is an AHA group then $\hbox{as--}\pi_i^\infty(G)=0$
for all $i\geqslant2$.
\end{twF}

In the same
section we study the relationship between the {\it asymptotic}
and the {\it ordinary} connectedness at infinity, showing that,
whenever it makes sense, the former implies the latter (Proposition \ref{aspi=pi}).
In particular, in view of this implication, one deduces Theorem D from
Theorem F, and this is our line of proof of Theorem D.

Section \ref{bdry} is devoted to the proof of the following.

\begin{twG}[see Theorem \ref{d1} in the text]
Let $X$ be a geodesic metric space which is AHA. Then
\begin{enumerate}
\item if $X$ is proper and $\delta$--hyperbolic then its Gromov boundary
contains no $2$--disk;
\item if $X$ is complete and CAT(0) then its boundary (equipped with the cone topology)
contains no $2$--disk.
\end{enumerate}
\end{twG}

Clearly, the conclusions of this theorem hold true for the corresponding types
of AHA groups, namely word hyperbolic ones and CAT(0) ones, respectively.
This conclusion should be contrasted with the fact that in both classes
there are examples with boundaries of arbitrary topological dimension.

In Subsection \ref{sysgap} we extend the techniques used to prove Theorem G,
and get the following analogous result for systolic complexes.

\begin{twH}[see Theorem \ref{d2} in the text]
Let $X$ be a systolic simplicial complex. Then its systolic boundary
(as defined in \cite{OP}) contains no $2$--disk. In particular, the systolic boundary
of any systolic group contains no $2$--disk.
\end{twH}


\section{Preliminaries}
\label{prel}

\subsection{Asymptotically hereditarily aspherical groups}
\label{AHA}

Let $(X,d)$ be a metric space. For $r>0$, the \emph{Rips complex} (with the constant $r$) of $X$, denoted $P_r(X)$, is a simplicial complex defined as follows. Vertices of $P_r(X)$ are points of $X$ and a set $A\subseteq X$ spans a simplex in $P_r(X)$ iff $\mr d(x,y)\leqslant r$, for every $x,y\in A$.

\begin{de}[AHA]
\label{daha}
(1) Given an integer $i\geqslant 0$, for subsets $C\subseteq D$ of a metric space $(X,d)$
we say that \emph{$C$ is $(i;r,R)$--aspherical in $D$} if every simplicial map $f\colon S\to P_r(C)$
(here and afterwards we always consider subsets as metric spaces with the restricted metric),
where $S$ is a triangulation of the $i$--sphere $S^i$, has a simplicial extension
$F\colon B \to P_R(D)$, for some triangulation $B$ of the $(i+1)$--ball $B^{i+1}$ such that $\partial B=S$.

(2) A metric space $X$ is \emph{asymptotically hereditarily aspherical}, shortly \emph{AHA},
if for every $r>0$ there exists $R>0$ such that every subset $A\subseteq X$ is $(i;r,R)$--aspherical in itself,
for every $i\geqslant 2$.
\end{de}

Let $g_1,g_2$ be real functions with $g_i(t)\stackrel{t\to \infty}{\longrightarrow} \infty$.
We say that a map $h\colon (X,d)\to (X',d')$ is a \emph{$(g_1,g_2)$--uniform embedding} if
\begin{align*}
g_1(d(x,y))\leqslant d'(h(x),h(y)) \leqslant g_2(d(x,y)),
\end{align*}
for every $x,y\in X$.

For $N>0$, a $(g_1,g_2)$--uniform embedding $h\colon X\to X'$ is a
\emph{$(g_1,g_2,N)$--uniform equivalence} between $X$ and $X'$ if $d(z,h(w))\leqslant N$,
for every $z\in X'$ and some $w\in X$.
A map $h\colon X\to X'$ is a \emph{uniform equivalence} (shortly \emph{u.e.})
if it is a $(g_1,g_2,N)$--uniform equivalence for some $g_1,g_2$ and $N$.
The spaces $X$ and $X'$ are then called \emph{uniformly equivalent} (shortly \emph{u.e.}).
Observe that for such an $h$ there is a \emph{coarsely inverse} map $h'\colon X'\to X$,
which is u.e.\ and for which there exists $M>0$ such that
$d(x,h'\circ h(x)),d'(y,h\circ h'(y))\leqslant M$, for every $x\in X$ and $y\in X'$.

In general, the uniform equivalence is weaker than the quasi-isometric equivalence.
However, if $h\colon (X,d)\to (X',d')$ is a u.e.\ between length spaces
(in particular between Cayley graphs of groups) then $(X,d)$ is quasi-isometric to $(X',d')$;
see e.g.\ \cite[Lemma 1.10]{Roe}.

It is proved in \cite[Proposition 3.2]{JS2} that if $h\colon X\to X'$ is a uniform embedding and if $X'$ is AHA, then $X$ is AHA. In particular AHA is a u.e.\ invariant and we define a finitely generated group $G$ to be \emph{asymptotically hereditarily aspherical (AHA)} if for some (and thus for every) word metric $d_S$ (induced by a finite generating set $S$) on $G$, the metric space $(G,d_S)$ is AHA.

{We finish this subsection with an easy observation leading to many examples of AHA groups. Recall that if a finitely
generated group $G$ acts (metrically) properly by isometries on a metric space $X$ then, for any $x\in X$, the orbit map $G\ni g\mapsto gx \in X$
is a uniform embedding (for $G$ equipped with a word metric).}
{
\begin{prop}
\label{propAHA}
A finitely generated group acting properly by isometries on an AHA metric space is AHA. 
\end{prop}
}
In particular, every finitely generated subgroup of an AHA group is AHA (see \cite[Corollary 3.4]{JS2}).

\subsection{Simplicial complexes}
\label{simpl}

Let $X$ be a simplicial complex. The $i$--skeleton of $X$ is denoted by $X^{(i)}$.
A subcomplex $Y$ of $X$ is \emph{full} if
every subset $A$ of vertices of $Y$ contained in a simplex of $X$, is
contained in a simplex of $Y$.
For a finite set $A=\lk v_1,\ldots,v_k \rk $ of vertices of $X$, by $\mr{span}(A)$ or by $\langle v_1,\ldots,v_k \rangle$ we denote the \emph{span} of $A$, i.e.\ the smallest full subcomplex of $X$ containing $A$.
A simplicial complex $X$ is \emph{flag} whenever every finite set of vertices of $X$ joined pairwise by edges in $X$, is contained in a simplex of $X$.
A \emph{link} of a simplex $\sigma$ of $X$ is a simplicial complex $X_{\sigma}=\lk \tau | \; \tau \in X \; \& \; \tau \cap \sigma=\emptyset \; \& \; \mr{span}(\tau \cup \sigma)\in X \rk$.

Let $k\geqslant 4$. A \emph{$k$--cycle} $(v_1,\ldots,v_{k})$ is a triangulation of a circle consisting of $k$ vertices: $v_1,\ldots,v_{k}$, and $k$ edges: $\langle v_i,v_{i+1}\rangle $ and $\langle v_k,v_{1}\rangle $.
A \emph{$k$--wheel (in $X$)} $(v_0;v_1,\ldots,v_k)$ (where $v_i$'s are vertices of
$X$) is a subcomplex of $X$ such that $(v_1,\ldots,v_k)$ is a full
$k$--cycle (i.e.\ a full subcomplex being a cycle) and $v_0$ is joined (by an edge in $X$) with $v_i$, for $i=1,\ldots,k$.
A \emph{$k$--wheel with a pendant triangle (in $X$)} $(v_0;v_1,\ldots,v_k;t)$ is a
subcomplex of $X$ being the union of a $k$--wheel $(v_0;v_1,\ldots,v_k)$ and a triangle
$\langle v_1,v_2,t \rangle \in X$, with $t\neq v_i$, $i=0,1,\ldots,k$.
A flag simplicial complex $X$ is \emph{$k$--large} if there are no $j$--cycles being full subcomplexes of $X$, for $j<k$.
$X$ is \emph{locally $k$--large} if all its links are $k$--large.

If not stated otherwise, speaking about a simplicial complex $X$,
we always consider a metric on the $0$--skeleton $X^{(0)}$,
defined as the number of edges in the shortest $1$--skeleton path joining two given vertices.
Given a nonnegative integer $i$ and a vertex $v\in X$, a (combinatorial) \emph{ball} $B_i(v)$
(respectively, \emph{sphere} $S_i(v)$) \emph{of radius $i$
around $v$} is a full subcomplex of $X$ spanned by vertices at distance
at most $i$ (respectively, at distance $i$) from $v$.

\subsection{(Weakly) systolic groups}
\label{weakp}

Recall -- see \cite{JS1} (or \cite{Che1}, where the name \emph{bridged complexes} is used) -- that a flag simplicial complex $X$ is called \emph{systolic} if it is simply connected and locally $6$--large.

In \cite{O-sdn}, the following class of \emph{weakly systolic complexes} was introduced as a generalization of
systolic complexes (cf.\ also \cite{CO}).

\begin{de}[Weakly systolic complex]
\label{weaks}
A flag simplicial complex $X$ is \emph {weakly systolic} if for every vertex $v$ of $X$
and for every $i=1,2,\ldots$ the following condition holds.
For every simplex $\sigma \subseteq S_{i+1}(v,X)$ the intersection $X_{\sigma}\cap B_i(v,X)$ is a single non-empty simplex.
\end{de}

Systolic complexes are weakly systolic \cite{JS1}. For every simply connected locally $5$--large cubical complex
(i.e.\ CAT(-1) cubical complex), there is a canonically associated weakly systolic simplicial complex,
to which any group action by automorphisms is induced \cite{O-sdn} (see Proposition~\ref{p:thcat-1} below).
It follows that the class of \emph{weakly systolic groups} (i.e.\ groups acting geometrically by simplicial automorphisms on weakly systolic complexes) contains essentially classes of \emph{systolic groups} (i.e.\ groups acting geometrically by simplicial automorphisms on systolic complexes) and ``CAT(-1) cubical groups''. For other classes of weakly systolic groups see \cite{O-sdn}.

The following result is an almost immediate consequence of the definition
of weakly systolic complexes, see \cite{O-sdn}.

\begin{lem}
\label{ws-contr}
Combinatorial balls in finite dimensional weakly systolic complexes are contractible. In particular,
any finite dimensional weakly systolic complex is contractible.
\end{lem}

Next definition presents a property of simplicial complexes which, together with
simple connectedness, characterizes weak systolicity. This definition, as well as
the next result, are taken from \cite{O-sdn}.

\begin{de}[$SD_2^{\ast}$ property]
\label{sd2*}
A flag simplicial complex $X$ satisfies the $SD_2^{\ast}$ property if the following two conditions hold.

(a) $X$ does not contain $4$--wheels,

(b) for every $5$--wheel with a pendant triangle $\widehat W$ in $X$, there exists a vertex $v$ with $\widehat W\subseteq B_1(v,X)$.
\end{de}

\begin{tw} [\cite{O-sdn}]
\label{logl}
A simplicial complex $X$ is weakly systolic if and only if
it is simply connected and satisfies the condition $SD_2^{\ast}$.
In particular, the universal cover of a simplicial complex satisfying the
condition $SD_2^{\ast}$ is weakly systolic. Consequently, complexes satisfying
the condition $SD_2^{\ast}$ are aspherical.
\end{tw}

In connection with the AHA property,
we are interested in a subclass of weakly systolic complexes
described by the following additional condition,
distinguished and studied in \cite{O-sdn}.

\begin{de}
\label{lasf}
A flag simplicial complex $X$ is called a \emph{complex with $SD_2^{\ast}$ links} if $X$ and every of its links satisfy the property $SD_2^{\ast}$.
\end{de}

Observe that, by Theorem \ref{logl}, the universal cover of a complex with $SD_2^{\ast}$ links, is weakly systolic. Moreover, every $6$--large simplicial complex (respectively, systolic complex) is a complex with $SD_2^{\ast}$ links (respectively, weakly systolic complex with $SD_2^{\ast}$ links). In Subsection \ref{chcg} we provide further (new) examples of high dimensional complexes of this type. For the construction we need the following fact, established in \cite{O-sdn}.

\begin{prop} [\cite{O-sdn}]
\label{fullasf}
Let $X$ be a finite-dimensional flag simplicial complex.
Then the following two conditions are equivalent.

i) $X$ is a complex with $SD_2^{\ast}$ links.

ii) Every full subcomplex of $X$ satisfies the $SD_2^{\ast}$ property.
\end{prop}

Weakly systolic complexes with $SD_2^{\ast}$ links seem to be (asymptotically)
closest to systolic complexes. One instance of this claim is the property
from the corollary below. To see it,
observe that, by Proposition \ref{fullasf}, every full subcomplex
of a finite dimensional weakly systolic complex with $SD_2^{\ast}$ links
has the property $SD_2^{\ast}$. By Theorem \ref{logl}, every such subcomplex
is aspherical. This shows the following.

\begin{cor} [\cite{O-sdn}]
\label{ws-asph}
Let $X$ be a finite dimensional weakly systolic simplicial complex with
$SD_2^{\ast}$ links. Then every full subcomplex of $X$ is aspherical.
\end{cor}

\subsection{Cubical complexes}
\label{cc}
\emph{Cubical complexes} are cell complexes in which every cell is isomorphic to a standard (Euclidean or hyperbolic) cube; see e.g.\ \cite[Appendix A]{Davis} for a precise definition. It means in particular that two cubes in a cubical complex intersect along a single subcube. The \emph{link} $Y_k$ of a cube $k$ in a cubical complex $Y$ is a simplicial complex defined in the following way.
Vertices of $Y_k$ are the minimal (with respect to the inclusion) cubes containing $k$ properly, and the set of such cubes spans a simplex of $Y_k$ if they are all contained in a common cube of $Y$.
A cubical complex is \emph{locally $k$--large} (respectively, \emph{locally flag}) if links of its vertices are $k$--large (respectively, flag).

A lemma of Gromov (see \cite[Appendix I]{Davis}) states that a simply connected locally flag (respectively, locally $5$--large) cubical complex admits a metric of non-positive (respectively, negative) curvature, or CAT(0) (respectively, CAT(-1)) metric.

\begin{de}[Thickening]
\label{thick}
Let $Y$ be a cubical complex. The \emph{thickening} $Th(Y)$ of $Y$ is
a simplicial complex defined in the following way. Vertices of
$Th(Y)$ are vertices of $Y$. Vertices $v_1,\ldots,v_k$
of $Th(Y)$ span a simplex iff they (as
vertices of $Y$) are contained in a common cube of $Y$.
\end{de}

The next lemma will be used later in this subsection and in Section~\ref{weak} (in the proof of Corollary~\ref{corCox}).

\begin{lem}[$SD_2^{\ast}$ links in thickenings]
\label{sd2th}
Let $Y$ be a cubical complex with links satisfying the $SD_2^{\ast}$ property. Then links in $Th(Y)$ satisfy the $SD_2^{\ast}$ property as well.
\end{lem}
\begin{proof}
First, we show condition (b) of the definition of the $SD_2^{\ast}$ property (Definition \ref{sd2*}).
For a given simplex $\sigma$ in $Th(Y)$ we have to show that every $5$--wheel with a pendant triangle $\widehat W$ in $Th(Y)_{\sigma}$ is contained in some ball of radius $1$ in $Th(Y)_{\sigma}$.
The proof follows the idea of the proof of \cite[Lemma 3.2]{O-chcg}, whose notation we adapt here.

Let $k$ be the minimal cube in $Y$ containing all vertices of $\sigma$ (treated as vertices of $Y$).
Recall that the $0$--skeleton of the link $Y_k$ is identified with the set of minimal cubes of $Y$ properly containing $k$.
For a vertex $v'\in Th(Y)_{\sigma}$, let $A_{v'}\subseteq Y_k^{(0)}$ be the set of all cubes in $Y_k^{(0)}$ belonging to the minimal cube containing $k$ and $v'$.
Let $\widehat W=(v_0;v_1,\ldots,v_5;t)$ be a $5$--wheel with a pendant triangle in $Th(Y)_{\sigma}$.
We will find a corresponding $5$--wheel with a pendant triangle $\widehat W'$ in $Y_k$. Since $Y_k$ satisfies the $SD_2^{\ast}$ property, $\widehat W'$ will be contained in some ball of radius $1$ in $Y_k$ and, consequently, $\widehat W$ will be contained in a ball of radius $1$ in $Th(Y)_{\sigma}$.

\medskip\noindent
{\bf Claim 1.} \emph{For $i\neq j$, $i\neq j\pm 1$ and $\{i,j\}\neq \{1,5\}$, there exist vertices (of $Y_k$) $z^i_j\in A_{v_i}$ and $z^j_i\in A_{v_j}$ not contained in a common cube containing $k$, and thus not connected by an edge in $Y_k$.}

\medskip\noindent
\emph{Proof of Claim 1:}
If every vertex from $A_{v_i}$ is connected by an edge (in $Y_k$) with any vertex from $A_{v_j}$ then, by the flagness of links in $Y$, all cubes from $A_{v_i}\cup A_{v_j}$ are contained in a common cube $K$ containing $k$. It follows then that $v_i,v_j \in K$ which contradicts the fact that the cycle $(v_1,v_2,\ldots,v_5)$ has no diagonal (since $\widehat W$ is a $5$--wheel with a pendant triangle in $Th(Y)_{\sigma}$.
\medskip

Observe that, since $k$ and $v_i,v_{i+1}$ (similarly $k,v_5,v_1$) are contained in common cube, we have that $\langle z^i_j, z^{i+1}_l \rangle \in Y_k$ (similarly $\langle z^5_j, z^{1}_l \rangle \in Y_k$), i.e.\ $z^i_j$ and $z^{i+1}_l$ are joined by an edge in $Y_k$.

\medskip\noindent
{\bf Claim 2.} \emph{$\langle z^1_3,z^4_2\rangle \notin Y_k$, $\langle z^2_4,z^5_3\rangle \notin Y_k$, and $\langle z^3_1,z^5_3\rangle \notin Y_k$.}

\medskip\noindent
\emph{Proof of Claim 2:}
If $\langle z^1_3,z^4_2\rangle \in Y_k$ then, by Claim 1, the $4$--cycle $(z^1_3,z^2_4,z^3_1,z^4_2)$
 has no diagonal, i.e.\ is a full subcomplex in $Y_k$.
For each $i=1,2,3,4$, consider a cube containing $k$ and $v_0,v_i$. Since all these cubes intersect, there exists a vertex $u\in Y_k$ contained in all of them. Then the complex $(u;z^1_3,z^2_4,z^3_1,z^4_2)$ is a $4$--wheel in $Y_k$, contrary to our assumptions.
This shows that $\langle z^1_3,z^4_2\rangle \notin Y_k$. Similarly, $\langle z^2_4,z^5_3\rangle \notin Y_k$. Consequently, an analogous argument shows that $\langle z^3_1,z^5_3\rangle \notin Y_k$. Thus the claim follows.
\medskip

From Claim 2 it follows that the cycle $c'=(z^1_3,z^2_4,z^3_1,z^4_2,z^5_3)$ is a $5$--cycle without diagonal.
Furthermore, for each $i=1,\ldots, 5$, consider a cube containing $k$ and $v_0,v_i$. Since all these cubes intersect, there exists a vertex $z\in Y_k$ contained in all of them. Clearly $(z;z^1_3,z^2_4,z^3_1,z^4_2,z^5_3)$ is a $5$--wheel in $Y_k$. Similarly we find a vertex $s\in A_t$, such that $\widehat W'=(z;z^1_3,z^2_4,z^3_1,z^4_2,z^5_3;s)$ is a $5$--wheel with a pendant triangle in $Y_k$. Then, by our assumptions on links in $Y$, there exists a vertex $w'\in Y_k$ with $\widehat W'\subseteq B_1(w',Y_k)$. It follows that there is a vertex $w\in Th(Y)_{\sigma}$ (any vertex of the cube $w'$ outside $k$ can be chosen as $w$) such that  $\widehat W \subseteq B_1(w,Th(Y)_{\sigma})$, hence condition (b) of Definition \ref{sd2*} holds.
\medskip

For condition (a) of the definition of the $SD_2^{\ast}$ property we proceed similarly. That is, assuming there is a $4$--wheel in a link of $Th(Y)$ we find a $4$--wheel in the corresponding link of $Y$ -- as done for $5$--wheels above.
This leads to a contradiction, establishing condition (a).
\end{proof}

The following result was proved in \cite{O-sdn}. The proof given there is in a large part very similar (though in a more general setting) to the proof of Lemma~\ref{sd2th} given above.

\begin{prop}[\cite{O-sdn}]
\label{p:thcat-1}
The thickening of a simply connected locally $5$--large (i.e.\ CAT($-1$)) cubical complex is a weakly systolic complex.
\end{prop}

Assume that a group $H$ acts freely by automorphisms on a cubical complex $Y$. This action induces an $H$--action on the thickening $Th(Y)$.
The \emph{minimal displacement} of the action of $H$ on $Th(Y)$ (in other words: the \emph{injectivity radius} of the quotient $H\backslash Th(Y)$)
is the number $\min \lk \mr d(v,hv) \; | \; v\in (Th(Y))^{(0)}, \; h\in H, \; h\ne1 \rk$.
Combining Lemma~\ref{sd2th} and Proposition~\ref{p:thcat-1} we obtain the following fact that will be useful in Section~\ref{weak}
(in the proof of Theorem~\ref{indCox}).

\begin{prop}
\label{p:quot}
Let $H$ act freely by automorphisms on a simply connected locally $5$--large cubical complex $Y$ whose links satisfy the $SD_2^{\ast}$ property.
Assume that the minimal displacement for the induced $H$--action on $Th(Y)$ is at least $5$. Then the quotient $H\backslash Th(Y)$ is a $5$--large complex with $SD_2^{\ast}$ links.
\end{prop}
\begin{proof}
By Proposition~\ref{p:thcat-1} the thickening $Th(Y)$ is weakly systolic. Weakly systolic complexes are $5$--large -- this follows easily from the definition (Definition~\ref{weaks}). Thus $Th(Y)$ is $5$--large and, by the assumptions on the injectivity radius, the quotient $H\backslash Th(Y)$ is $5$--large as well.
Since $H\backslash Th(Y)$ is locally isomorphic to $Th(Y)$, by Lemma~\ref{sd2th}, we obtain that links in the quotient satisfy the
$SD_2^{\ast}$ property.
Similarly for a $5$--wheel with a pendant triangle $\widehat W$ in $H\backslash Th(Y)$, by the local isomorphism and by the assumptions
on the injectivity radius one can find the corresponding $5$--wheel $\widehat W'$ with a pendant triangle in $Th(Y)$. Since $Th(Y)$ (as weakly systolic)
satisfies the $SD_2^{\ast}$ property (Theorem~\ref{logl}), $\widehat W'$ is contained in a ball of radius one. It implies that $\widehat W$ is contained in a corresponding ball of radius one.
This establishes condition (b) from Definition~\ref{sd2*} for $H\backslash Th(Y)$. As for condition (a), observe that (by local isomorphism) the existence of a $4$--wheel in $H\backslash Th(Y)$ would imply the same for $Th(Y)$.
In view of Theorem~\ref{logl}, this would contradict weak systolicity of $Th(Y)$.
\end{proof}

\subsection{Finiteness properties of groups}
\label{fp}
Recall -- see e.g.\ \cite[Chapter 7.2]{Geo} -- that a group $G$ is said to have \emph{type $F_n$} ($n\geqslant 1$) if there exists a $K(G,1)$ CW complex having finite $n$--skeleton. In particular, a group $G$ has type $F_1$ iff $G$ is finitely generated and $G$ has type $F_2$ iff it is finitely presented.
We say that $G$ has \emph{type $F_{\infty}$} if there exists a $K(G,1)$ complex whose all skeleta are finite. It is a standard fact -- see \cite[Proposition 7.2.2]{Geo} -- that a group $G$ has type $F_{\infty}$ iff it has type $F_n$ for every $n\geqslant 1$. In particular, every finite group has type $F_{\infty}$ -- see \cite[Corollary 7.2.5]{Geo}.

Let $G$ be a group acting by automorphisms on a path connected CW complex $X$, i.e.\ $X$ is a \emph{$G$--CW complex} -- see \cite[Chapter 3.2]{Geo}. A \emph{$G$--filtration} of $X$ is a countable collection
of $G$--subcomplexes $X_0\subseteq X_1 \subseteq \cdots$
such that $X=\bigcup_{i=0}^{\infty} X_i$. Let $X$ be $(n-1)$--connected. The $G$--filtration $\lk X_i \rk$ is called \emph{essentially $(n-1)$--connected} if for $0\leqslant k \leqslant n-1$
and for every $i$ there exists $j>i$ such that the map
$\pi_k(X_i,y)\to \pi_k(X_j,y)$
induced by the inclusion is trivial (here $y$ is a fixed point in $X_0$).

Recall the following Brown's criterion -- see \cite[Theorems 2.2 and 3.2]{Br} and \cite[Theorem 7.4.1 and the excercise after it]{Geo}.

\begin{tw}[Brown's Finiteness Criterion]
\label{BFC}
Let $X$ be an $(n-1)$--connected $G$--CW complex such that for every
$i$--cell $c$ of $X$ its $G$--stabilizer $G_c$ has type $F_{n-i}$.
Assume that $X$ admits a $G$--filtration $\lk X_i \rk$ where each $G\backslash X_i$ has finite $n$--skeleton. Then $G$ has type $F_n$ iff $\lk X_i \rk$ is essentially $(n-1)$--connected.
\end{tw}

The following fact is well known -- see e.g.\ \cite[Lemma 1]{FO}.

\begin{lem}[Simply connected Rips complex]
\label{scRc}
Let $G$ be a finitely presented group with a finite generating set $S$. There
exists a number $D$ such that for $d\geqslant D$ the Rips complex $P_d(G)$ (of $G$ with word metric with respect to $S$) is simply connected.
\end{lem}

Recall -- see e.g.\ \cite[Chapter VIII]{Bro} -- that the \emph{cohomological dimension} of a group $G$ is defined as
\begin{align*}
\mr{cd}\, G=\sup \lk i|\; H^i(G,M)\neq 0\; \mr{for}\; \mr{some}\; G\mr{-module}\; M \rk.
\end{align*}
If $\mr{cd}\, G<\infty$ then $G$ is torsion-free. The \emph{virtual cohomological dimension} of a virtually torsion-free group $G$, denoted \emph{$\mr{vcd}\, G$}, is defined as  $\mr{vcd}\, G=\mr{cd}\, H$, for some (and hence any) torsion-free finite index subgroup $H$ of $G$.

\subsection{Vietoris homology}
\label{vietpr}
We recall basic notions and facts related to the
concept of Vietoris homology,
which we use in Section \ref{bdry}.
Our main reference is the book \cite{Le}
by Solomon Lefschetz. If not stated otherwise, we use integers $\mathbb Z$ as coefficients.

Given a set $X$, let $V(X)$ denote the simplicial complex with vertex set $X$,
such that any finite subset of $X$ spans a simplex. Let $S(X)$ be the set of
all simplices of $V(X)$; we identify simplices of $S(X)$ with their vertex sets.
A {\it diameter function} in $X$ is any function $\nu:S(X)\to[0,\infty)$ such that:
\begin{enumerate}
\item
$\nu(\sigma)=0$ for all $0$--simplices $\sigma$, and
\item
if $\tau$ is a face
of $\sigma$ then $\nu(\tau)\leqslant \nu(\sigma)$.
\end{enumerate}

\noindent
{\bf Example.} If $(X,d)$ is a metric space, then the function
$\nu_d:S(X)\to[0,\infty)$ defined by
$\nu_d(\sigma):=\max\{ d(x,y):x,y\in\sigma \}$ is a diameter function in $X$.
We will call $\nu_d$ the diameter function {\it associated to the metric} $d$.

\medskip
Given a set $X$ with a diameter function $\nu$, and a number $\varepsilon>0$,
the {\it Vietoris complex} $V_\varepsilon(X,\nu)$ is the subcomplex of $V(X)$
consisting of all those $\sigma\in S(X)$ for which $\nu(\sigma)\leqslant\varepsilon$.

\medskip\noindent
{\bf Example.} If $\nu$ is the diameter function in $X$ associated to a metric $d$
then for any $\varepsilon\geqslant 0$ the Vietoris complex $V_\varepsilon(X,\nu)$
coincides with the Rips complex $P_\varepsilon(X,d)$.

\medskip
Denote by $Z_k(V_\varepsilon(X,\nu))$ the set of $k$--dimensional simplicial
cycles in the Vietoris complex $V_\varepsilon(X,\nu)$.

\begin{de}[$V$--cycle]
\label{v1}
A \emph{$k$--dimensional $V$--cycle in $(X,\nu)$} is a sequence
$(z_n)$ of simplicial $k$--dimensional cycles in $V(X)$ such that for some
sequence $\varepsilon_n\to0$ of positive numbers we have
\begin{enumerate}
\item $z_n\in Z_k(V_{\varepsilon_n}(X,\nu))$ for all $n$;
\item $z_n$ is homologous to $z_{n+1}$ in $V_{\varepsilon_n}(X,\nu)$,
for all $n$.
\end{enumerate}
\end{de}

The set of all $k$--dimensional $V$--cycles in $(X,\nu)$ clearly forms an abelian group
with respect to the coordinate-wise addition $(z_n)+(z'_n)=(z_n+z'_n)$.
We denote this group by $VZ_k(X,\nu)$.

\begin{de}[$V$--boundary]
\label{v2}
A $k$--dimensional $V$--cycle $(z_n)$ in $X$ is a \emph{$V$--boundary}
if there is a sequence $w_n$ of $(k+1)$--dimensional simplicial chains in $V(X)$ such that for some sequence
$\varepsilon_n\to0$ of positive numbers, $w_n$ is a chain in $V_{\varepsilon_n}(X,\nu)$
and $z_n=\partial w_n$.
\end{de}

The set of all $k$--dimensional boundaries clearly forms a subgroup in $VZ_k(X,\nu)$,
and we denote it $VB_k(X,\nu)$.

\begin{de}[Vietoris homology]
\label{v3}
\emph{Vietoris homology} of a space with diameter function $(X,\nu)$ is the
quotient group $VH_*(X,\nu):=VZ_*(X,\nu)/ VB_*(X,\nu)$. Vietoris homology
of a metric space $(X,d)$ is the Vietoris homology of $(X,\nu_d)$.
\end{de}

We recall from \cite{Le} some properties of Vietoris homology.
Two diameter functions $\nu,\nu':S(X)\to[0,\infty)$ are {\it equivalent}
if for any sequence $(\sigma_n)$ of simplices we have $\nu(\sigma_n)\to 0$ iff $\nu'(\sigma_n)\to0$.
Equivalently, $\nu$ and $\nu'$ are equivalent if $\forall\varepsilon$ $\exists\delta$
$\forall\sigma\in S(X)$ if $\nu(\sigma)<\delta$ then $\nu'(\sigma)<\varepsilon$
and if $\nu'(\sigma)<\delta$ then $\nu(\sigma)<\varepsilon$.

\begin{lem}[{\cite[statement (27.4) on p.\ 242]{Le}}]
\label{v4}
If $\nu,\nu'$ are equivalent
diameter functions in $X$ then $VH_*(X,\nu)=VH_*(X,\nu')$.
\end{lem}

\begin{cor}
\label{v5}
Let $X$ be a compact metric space and let $d_1,d_2$ be any two metrics
on $X$ compatible with the topology. Then the diameter functions $\nu_{d_1},\nu_{d_2}$
are equivalent. Consequently, we have $VH_*(X,\nu_{d_1})=VH_*(X,\nu_{d_2})$, and thus
the Vietoris homology is a topological invariant of a compact metric space.
\end{cor}

The next result identifies Vietoris homology with other homology theories.

\begin{tw}[{\cite[Theorem (26.1) on p.\ 273]{Le}}]
\label{v6}
For any compact metric space $X$ its Vietoris homology groups coincide with
\v Cech homology groups. In particular, if $X$ is a compact CW complex then its
Vietoris homology coincides with its singular homology.
\end{tw}

We finish this section with an observation which will be used
to justify Claim 1 in the proof of
Proposition \ref{d4}. This observation is a straightforward consequence of the definition
of the $V$--boundary.

\begin{fact}
\label{v7}
Suppose that $(z_n)$ is a $V$--cycle which is homologically
nontrivial in $VH_*(X,\nu)$. Then $\exists a>0$ and $\exists n_0\in N$ such
that $\forall n\geqslant n_0$ the cycle $z_n$ is homologically nontrivial in $V_a(X,\nu)$
(i.e.\ induces a nontrivial element in the ordinary simplicial homology
of the simplicial complex $V_a(X,\nu)$).
\end{fact}

\section{A criterion for AHA}
\label{scrit}

Before formulating a criterion for AHA, we need to discuss assumptions
that we put on metric spaces to which the criterion applies.
We will deal with simply connected geodesic metric spaces $X$
for which the filling radius function has finite values for all arguments.
More precisely, given a piecewise geodesic loop $f\colon S^1\to X$ and its
continuous extension $F\colon D^2\to X$, the \emph{filling radius} of $F$, denoted
fill-rad$(F)$, is the maximal distance of a point in $F(D^2)$ from $f(S^1)$.
Now, we define fill-rad$(f)$ to be the infimum over all $F$ as above
of the values fill-rad$(F)$. Finally, given a real number $r$, we put
fill-rad$(r)$ to be the supremum over all loops $f$ as above
with length $\leqslant r$ of the values fill-rad$(f)$. This gives us a function
$\hbox{fill-rad}\colon R_+\to R\cup\{\infty\}$, which we call the \emph{filling radius
function} for $X$.

\medskip
\rem The condition that fill-rad$(r)<\infty$ for each $r>0$
is automatically satisfied by simply connected metric complexes $X$ acted upon geometrically
by a group of combinatorial automorphisms. Indeed, 
this is easily implied by the following two observations:
\begin{itemize}
\item for each finite subcomplex $K$ of $X$, there is another finite subcomplex
$L_K$ containing $K$ such that the map
$\pi_1K\to\pi_1L_K$ induced by the inclusion is trivial;
\item loops
of bounded length are contained in finitely many (up to the group action) finite simply connected subcomplexes.
\end{itemize}

\medskip
Given a metric space $(X,d)$, a subset $A\subseteq X$ and a real number $c>0$,
put $N_c(A)=\{ x\in X|\; d(x,A)\leqslant c \}$.
A criterion for AHA is stated in the following.

\begin{prop}
\label{crit}
Let $X$ be a simply connected geodesic metric space
with fill-rad$(r)<\infty$ for each $r>0$. Suppose also that to each subset
$A\subseteq X$ there is assigned a subspace $Y_A$ of $X$ such that:
\begin{itemize}
\item $A\subseteq Y_A$ and $Y_A\subseteq N_D(A)$ for some universal
$D$ independent of $A$;
\item $Y_A$ is aspherical;
\item if $A_1\subseteq A_2$ then $Y_{A_1}\subseteq Y_{A_2}$.
\end{itemize}
Then $X$ is AHA.
\end{prop}
\begin{proof}
Fix any subset $A\subseteq X$, a real number $r>0$ and a simplicial map
$f\colon S\to P_r(A)$ for some triangulation $S$ of a $k$--sphere, with $k\geqslant 2$.
View $f$ as the map from the vertex set $S^{(0)}$ to $A$.

\medskip
\noindent
{\bf Claim 1.} There exists a continuous map $\bar f\colon S\to X$ and its continuous
extension $\bar F\colon D^{k+1}\to X$ such that:
\begin{enumerate}
\item $\bar f(v)=f(v)$
for any vertex $v\in S^{(0)}$,
\item for any simplex $\sigma$ of $S$ it holds
$\bar f(\sigma)\subseteq N_C(f(\sigma^{(0)}))$, and
\item $\bar F(D^{k+1})\subseteq N_C(A)$,
\end{enumerate}
where $C=r/2+\hbox{fill-rad}(3r)+D$ and $\sigma^{(0)}$ denotes the vertex set
of $\sigma$.

\medskip
To get the claim, we turn to constructing $\bar f$, by successive extensions
over skeleta of $S$. Condition (1) determines $\bar f$ on the $0$--skeleton $S^{(0)}$.
Extend $\bar f$ to the $1$--skeleton $S^{(1)}$ by connecting the images of adjacent
vertices with geodesic segments. These segments have lengths $\leqslant r$, so we get
$\bar f(\sigma^{(1)})\subseteq N_{r/2}(f(\sigma^{(0)}))$ for every simplex $\sigma$
of $S$. Next, we extend $\bar f$ to the $2$--skeleton $S^{(2)}$ as follows. For each
$2$--simplex $\tau$ of $S$, the image $\bar f(\partial\tau)$ is a geodesic loop
of length $\leqslant 3r$. Thus we may (and do) extend $\bar f|_{\partial\tau}$ over $\tau$
so that $\bar f(\tau)\subseteq N_{{\rm fill-rad}(3r)}(\bar f(\partial\tau))
\subseteq N_{r/2+{\rm fill-rad}(3r)}(f(\tau^{(0)}))$.

Further extensions require referring to the assumptions
concerning aspherical subsets $Y_A$. Put $c=r/2+\hbox{fill-rad(3r)}$.
Given a $3$--simplex $\rho$, consider the set $Y_{\bar f(\partial\rho)}$.
Since this set is aspherical, we may (and do) extend $\bar f$ over $\rho$
so that $\bar f(\rho)\subseteq Y_{\bar f(\partial\rho)}\subseteq N_{c+D}(f(\rho^{(0)}))
=N_C(f(\rho^{(0)}))$. If $\rho$ is a simplex of dimension higher than 3,
and $\bar f$ is already defined on $\partial\rho$, consider the set
$Y_{\rho^{(2)}}$. Since for codimension one faces $\sigma$ of $\rho$
we have the inclusions (the first of which is a part of inductive hypothesis)
$\bar f(\sigma)\subseteq Y_{\bar f(\sigma^{(2)})}\subseteq Y_{\bar f(\rho^{(2)})}$, we may
extend $\bar f$ over $\rho$ so that
$\bar f(\rho)\subseteq Y_{\bar f(\rho^{(2)})}\subseteq N_C(f(\rho^{(0)}))$.
This gives us $\bar f$ as required.

To get $\bar F$, consider the set $Y_{\bar f(S^{(2)})}$ and note that
$\bar f(S)\subseteq Y_{\bar f(S^{(2)})}$, by the third condition in the assumptions
of the proposition. By asphericity of the latter set
there exists an extension $\bar F\colon D^{k+1}\to Y_{\bar f(S^{(2)})}$.
Since $Y_{\bar f(S^{(2)})}\subseteq N_D(\bar f(S^{(2)}))\subseteq N_C(f(S^{(0)}))$,
Claim 1 follows.

Now, observe that by the above construction the image under $\bar f$
of any simplex of $S$ has diameter $\leqslant r+2C$. Thus, for $K=r+2C+1=2(r+\hbox{fill-rad}(3r)+D)+1$
we have the following.

\medskip\noindent
{\bf Claim 2.} There exists a triangulation $B$ of the disk $D^{k+1}$,
which coincides with the triangulation $S$ when restricted to $S=\partial D^{k+1}$,
such that the image of each simplex of $B$ under $\bar F$ has diameter $\leqslant K$.

\medskip
We will now use a triangulation $B$ as in Claim 2, and the map $\bar F$, to describe
a map $F\colon B\to P_R(A)$ as required in the definition of AHA. We will specify the
(universal) value of $R$ at the end of this description.

It is sufficient to describe values of $F$ at vertices of $B$. For vertices
$v$ of $S$ we put $F(v)=f(v)$.
For vertices $v$ in $B$ and not in $S$ choose $F(v)$ to be an element of $A$
at distance at most $C$ from $\bar F(v)$. Clearly, since
$\bar F(D^{k+1})\subseteq N_C(f(S^{(0)}))$ and $f(S^{(0)})\subseteq A$, such choices
are possible. It is also clear that, after such choices, distances between
images of adjacent vertices of $B$ are bounded by $K+2C=3r+4(\hbox{fill-rad}(3r)+D)+1$.
Thus, putting $R=3r+4(\hbox{fill-rad}(3r)+D)+1$, we get the simplicial map
$F\colon B\to P_R(A)$ as required, and the proposition follows.
\end{proof}

\medskip
In view of  
Proposition \ref{propAHA} and
the remark right before  Proposition \ref{crit},
the latter proposition has the following two corollaries.

\begin{cor}
\label{ccrit}
Let $G$ be a group acting {properly by isometries}
on a simply connected geodesic metric space $X$ with fill-rad$(r)<\infty$, for each $r>0$.
Suppose also that to each subset
$A\subseteq X$ there is assigned a subspace $Y_A$ of $X$ such that:

\begin{itemize}
\item $A\subseteq Y_A$ and $Y_A\subseteq N_D(A)$ for some universal
real number $D$ independent of $A$;
\item $Y_A$ is aspherical;
\item if $A_1\subseteq A_2$ then $Y_{A_1}\subseteq Y_{A_2}$.
\end{itemize}
Then $G$ is AHA.
\end{cor}

\medskip
We say that a group acts on a geodesic metric space {\it geometrically}
if it acts by isometries, properly discontinuously and cocompactly.
A geometric action is automatically proper.

\begin{cor}
\label{ccrit+}
Let $G$ be a group acting geometrically, by combinatorial automorphisms,
on a simply connected geodesic metric complex $X$.
Suppose also that to each subset
$A\subseteq X$ there is assigned a subcomplex $Y_A$ of $X$ such that:

\begin{itemize}
\item $A\subseteq Y_A$ and $Y_A\subseteq N_D(A)$ for some universal
real number $D$ independent of $A$;
\item $Y_A$ is aspherical;
\item if $A_1\subseteq A_2$ then $Y_{A_1}\subseteq Y_{A_2}$.
\end{itemize}
Then $G$ is AHA.
\end{cor}

\section{Groups acting on hereditarily aspherical $2$--complexes}
\label{twocpl}

\medskip
A celebrated Whitehead's conjecture asserts that each subcomplex of an aspherical
$2$--dimensional complex is aspherical. Up to date, this conjecture remains not proved
(or disproved). We will call each aspherical $2$--complex satisfying this conjecture
{\it hereditarily aspherical}.

\begin{lem}
\label{asf2}
Let $X$ be a
simply connected
hereditarily aspherical geodesic metric $2$--complex $X$ with fill-rad$(r)<\infty$, for each $r>0$, and whose cells have uniformly bounded diameter. Let $G$ be a group acting properly by isometries on $X$.
Then $X$ and, consequently, $G$ are AHA.
In particular, any group acting {geometrically} on a simply connected
hereditarily aspherical $2$--complex $X$ is AHA.
\end{lem}
\begin{proof}
In view of Proposition~\ref{crit}, and Corollaries~\ref{ccrit} and \ref{ccrit+}, the lemma follows by taking $Y_A$
to be the smallest subcomplex of $X$ containing $A$.
\end{proof}

\begin{cor}
\label{cat02}
Groups acting {properly by isometries} on CAT(0) $2$--complexes
with uniformly bounded diameters of cells are AHA.
\end{cor}
\begin{proof}
{Observe that for CAT(0) complexes fill-rad$(r)<\infty$, for each $r>0$.}
Since every subcomplex of a $2$--dimensional CAT(0) complex
is nonpositively curved, it follows that it is aspherical. Thus,
CAT(0) $2$--complexes are hereditarily aspherical. Let $X$ be such a complex,
satisfying the assumptions of the corollary. By uniform boundedness
of cells, if we associate to any subset $A\subset X$ 
the smallest subcomplex $Y_A$ of $X$ containing $A$,
then the conditions of Corollary \ref{ccrit} for $Y_A$ are satisfied.
The assertion follows then
by {Corollary~\ref{ccrit}}.
\end{proof}
\medskip

Next result is a useful special case of Lemma \ref{asf2}.

\begin{lem}
\label{prescom}
The fundamental group $G$ of any finite hereditarily aspherical
complex $K$ is AHA. In particular,
if the presentation complex of a finitely presented group $G$
(for some finite presentation) is hereditarily aspherical then $G$ is AHA.
\end{lem}
\begin{proof}
The group $G$ acts geometrically on the universal cover $X$ of $K$, which
is clearly simply connected. Since hereditary asphericity
passes to covers (see \cite[Theorem B, p.\ 28]{Hu}),
it follows that $X$ is hereditarily
aspherical. Applying Lemma \ref{asf2} concludes the proof.
\end{proof}

\medskip
Various classes of groups with hereditarily aspherical presentation complexes
occur in the literature. We recall some of them below. By Lemma \ref{prescom}, all these
groups are AHA.

\begin{exs}
\label{4.4}
(1) {\it $1$--relator groups.} By a result of R.\ Lyndon (see \cite[Proposition III.11.1]{LS}), the presentation complex of a $1$--relator group is
aspherical. Moreover, since this complex is the wedge of few circles
(possibly none) and the subcomplex spanned by the unique $2$--cell, all its subcomplexes are aspherical. It follows that presentation
complexes of $1$--relator groups are hereditarily aspherical.

(2) {\it Knot groups.} Theorem 5.8 of \cite{CCH}, which applies to knot
complements by the comment on p.\ 34 of the same paper, shows that knot
groups satisfy a certain condition CLA stronger than asphericity
of some presentation complex of a group. By \cite[Proposition 2.4]{CCH}, CLA passes
to sub-presentations, from which it follows that any CLA presentation
complex is hereditarily aspherical. In particular, all knot groups
have presentation complexes that are hereditarily aspherical.
The same is true for other fundamental groups of $3$--manifolds
with non-empty boundary occurring in \cite[Theorem 5.8]{CCH}.

(3) {\it Small cancellation groups (finitely presented).} 
It is known that small cancellation
groups do not admit reduced spherical van Kampen diagrams. It follows that
their presentation complexes $X$ are aspherical. Moreover, any subcomplex of
any such $X$ is itself the presentation complex of a small cancellation group,
hence it is also aspherical.

The above observation generalizes to finitely presented 
graphical small cancellation groups
as described in \cite{Oll2}, \cite{Gru}. Namely, it is shown in these papers
that a
presentation complex of any such group, for any set of relations corresponding
to standard generating cycles of the fundamental group of the
involved graph, is aspherical. Moreover, any subcomplex of a presentation
complex as above is easily seen to be a presentation complex of a similar form,
for a subgraph of the initial graph, and hence it is also aspherical.

(4) {\it Random groups (in the sense of Gromov).}
By the fact that the linear isoperimetric inequality holds in random groups
for all reduced van Kampen diagrams, these groups do not admit spherical
reduced van Kampen diagrams, compare \cite[Section V.c]{Oll}. Since these properties
pass to sub-presentations obtained by deleting some relations,
the initial presentation complex is hereditarily aspherical.
\end{exs}

\rem
AHA groups resulting from Lemma \ref{prescom} have cohomological dimension $\leqslant 2$.
It is not known whether their asymptotic dimension is $2$ (compare \cite{Dra},
Problems 3.15(a) and 3.16).
\medskip

Lemma \ref{prescom} can be extended to some groups that are not necessarily
finitely presented, see Lemma \ref{prescom'} below. Note that this lemma does not
follow from Corollary \ref{ccrit}. Indeed, let $G$ be a group as in Lemma
\ref{prescom'}, and suppose it is not finitely presented.
Then for any metric on the universal cover
of the presentation complex of $G$, for which
the action of $G$ is isometric and proper, the diameters of the 2-cells are not
uniformly bounded; consequently, the smallest subcomplex containing a subset $A$ does not lie within a uniformly bounded distance from $A$,
and thus it is not a good candidate for a subspace $Y_A$ as required
in the assumptions of Corollary \ref{ccrit}.
Nevertheless, the lemma below follows by another argument.

\begin{lem}
\label{prescom'}
 Let $G$ be a finitely generated group whose
presentation complex (for some presentation with finite generating set) is
hereditarily aspherical.
Then $G$ is AHA.
\end{lem}
\begin{proof}
The proof repeats the line of arguments from the proof of Proposition \ref{crit},
with some necessary modification. We begin with a preparatory observation which
leads to this modification.

Denote by $K(G)$ the presentation complex of $G=\langle S\,|\,R \rangle$,
where $S$ is finite. We assume that this complex is hereditarily aspherical.
Let $p:\widetilde K(G)\to K(G)$ be the universal covering of $K(G)$.
Given $L>0$, let $K_{\leqslant L}(G)$ be the subcomplex of $K(G)$ whose $1$--skeleton
coincides with that of $K(G)$, and which contains those $2$--cells of $K(G)$
which correspond to the relators from $R$ of length $\leqslant L$.
Observe that for any $L$ the subcomplex $K_{\leqslant L}(G)$ is finite and hereditarily
aspherical. Put also $\widetilde K_{\leqslant L}(G):=p^{-1}(K_{\leqslant L}(G))$.
Since the restriction of $p$ to $\widetilde K_{\leqslant L}(G)$ is a covering map
on $K_{\leqslant L}(G)$, and since hereditary asphericity passes to covers,
the complex $\widetilde K_{\leqslant L}(G)$ is also hereditarily aspherical.
The reader should also keep in mind that the $1$--skeleton of the complex
$\widetilde K_{\leqslant L}(G)$ coincides with the Cayley graph $C(G,S)$.

\medskip\noindent
{\bf Claim 0.} {\sl For each $r>0$ there is $L=L(r)$ such that any closed polygonal curve
of length $\leqslant 3r$ in the Cayley graph $C(G,S)$ is contractible in the complex
$\widetilde K_{\leqslant L}(G)$. Moreover, there is a number $F=F(r)$ such that any curve
$\gamma$ as above is contractible in the $F$--neighborhood of $\gamma$
in $\widetilde K_{\leqslant L}(G)$.}

\medskip
Both parts of Claim 0 easily follow from the fact that, up to $G$--action,
there are only finitely many closed curves $\gamma$ as above. Note that the constant $F$
in the claim corresponds to the filling radius parameter fill-rad$(3r)$.

For any subset $Z\subseteq \widetilde K_{\leqslant L}(G)$ let $Y_Z$ be the smallest subcomplex
of $\widetilde K_{\leqslant L}(G)$ containing $Z$. By what was said above, any such $Y_Z$
is aspherical. Let $D=D(L)$ be a number such
that for any $Z$ we have $Y_Z\subseteq N_D(Z)$. Such $D$ exists as, up to $G$--invariance,
we have finitely many types of $2$--cells in $\widetilde K_{\leqslant L}(G)$.

Now we get in the scheme of the proof of Proposition \ref{crit}.
Fix any subset $A\subseteq G$, a real number $r>0$ and a simplicial map
$f\colon S\to P_r(A)$ for some triangulation $S$ of a $k$--sphere, with $k\geqslant2$.
View $f$ as the map from the vertex set $S^{(0)}$ to $G$.
Let $L=L(r)$ and $F=F(r)$ be the numbers as in Claim 0, and let $D=D(L)$
be as in the previous paragraph. By referring to Claim 0, we can repeat the arguments as in the proof of
Proposition \ref{crit} to get the following claim, which is analogous to Claim 1
in the latter proof.

\medskip\noindent
{\bf Claim 1.} {\sl There exists a continuous map $\bar f\colon S\to \widetilde K_{\leqslant L}(G)$
and its continuous
extension $\bar F\colon D^{k+1}\to X$ such that:
\begin{enumerate}
\item $\bar f(v)=f(v)$
for any vertex $v\in S^{(0)}$,
\item for any simplex $\sigma$ of $S$ it holds
$\bar f(\sigma)\subseteq N_C(f(\sigma^{(0)}))$, and
\item $\bar F(D^{k+1})\subseteq N_C(A)$,
\end{enumerate}
where $C=r/2+F+D$ and $\sigma^{(0)}$ denotes the vertex set
of $\sigma$.}

\medskip
The remaining part of the proof is then exactly as that of Proposition \ref{crit}.
We omit further details.
\end{proof}

\medskip
\rems
(1) Note that an infinite CW complex is hereditarily aspherical iff all its finite
subcomplexes are aspherical.

(2) In view of (1), examples of groups which are not finitely presented, and to which Lemma \ref{prescom'}
applies, are small cancellation groups (also graphical ones) with infinitely many relators --
compare Example~\ref{4.4}(3) above.

\medskip
We finish the section with the statement of a problem which naturally occurs
in the context of this section (especially, Lemmas~\ref{prescom} and \ref{prescom'}),
and which may be viewed as a coarse group theoretic variant
of Whitehead's conjecture.

\medskip\noindent
{\bf Problem.} In accordance with Definition 2.1 (of the AHA property),
we say that a metric space $X$ is {\it asymptotically aspherical} if for each $r>0$
there is $R>0$ such that $X$ is $(i;r,R)$--aspherical in itself for every $i\geqslant2$.
Clearly, asymptotic asphericity is a coarse invariant. We pose the following problem:
{\it is every finitely generated $2$--dimensional and asymptotically aspherical group AHA?}

\medskip
In the above statement, being $2$--dimensional refers to any reasonable sense
of this notion, e.g.\ to (virtual) cohomological dimension, asymptotic dimension,
$1$--dimensional boundary (e.g.\ for word hyperbolic or CAT(0) groups), etc. Note also that
the assumption of asymptotic asphericity is necessary, in the sense that it must be
satisfied by any AHA group; on the other hand, we do not know
any examples of $2$--dimensional groups that are not asymptotically aspherical.

Observe that if the original Whitehead's conjecture holds then,
due to Lemmas \ref{asf2}, \ref{prescom} and \ref{prescom'}, the above problem
has positive answer for fundamental groups of $2$--dimensional
aspherical complexes (or more generally, for groups acting geometrically
by automorphisms on contractible $2$--complexes).

Compare the above statement
of coarse Whitehead conjecture with its other version given by
M.\ Kapovich -- see \cite[Problems 107 and 108]{Kap}.

\section{Weak systolicity and AHA}
\label{weak}

All the examples of finitely presented AHA groups provided in Section \ref{twocpl} are
groups of dimension at most two. It is an intriguing phenomenon that
AHA groups exist in every dimension (here ``dimension'' means e.g.\
(virtual) cohomological dimension or asymptotic dimension, but other
notions of a ``dimension of a group'' can be also considered).
The opportunity of constructing such groups is provided by
the following theorem whose special case concerning the
systolic groups was first proved in \cite{JS2}.

\begin{tw}
 \label{systaha}
A group $G$ acting {properly by isometries} on a finite dimensional weakly systolic complex with
$SD_2^{\ast}$ links is AHA. In particular systolic groups are
AHA.
\end{tw}
\begin{proof}
 Let $X$ be a weakly systolic complex with
$SD_2^{\ast}$ links, on which $G$ acts. By the definition of weak
systolicity, $X$ is
contractible, and by Corollary \ref{ws-asph}, every full subcomplex
of $X$ is aspherical. 
By Lemma~\ref{ws-contr} combinatorial balls in $X$ are contractible, and hence fill-rad$(r)<\infty$, for each $r>0$.
Thus, by Corollary \ref{ccrit}, $X$ is AHA, by
taking $Y_A$ to be the smallest full subcomplex of $X$ containing $A$,
for every subset $A\subseteq X$.
\end{proof}

\rem Note that it is not true in general that a weakly systolic complex is AHA.
The CAT(-1) cubical complex $Y$
dual to the tessellation of the hyperbolic $3$--space by regular right-angled dodecahedra is not AHA, by \cite[Corollary 6.2]{JS2}.
Consequently, the thickening $Th(Y)$ is weakly systolic (by Proposition~\ref{p:thcat-1}) but not AHA.
\medskip

The first examples of high dimensional systolic groups were
constructed in \cites{JS0, JS1}. In \cite{O-chcg} a simple
construction of high dimensional weakly systolic (systolic
and non-systolic) groups is provided. Below we present a construction
of high dimensional AHA right-angled Coxeter groups based on results
from \cite{O-chcg}. More precisely, the groups that we obtain
act geometrically on weakly systolic complexes with $SD_2^{\ast}$ links,
thus being AHA in view of Theorem \ref{systaha}.
\medskip

\rem
Constructions from \cites{JS0, JS1, O-chcg} provide examples of
(virtually) torsion-free hyperbolic groups of (virtual) cohomological
dimension $n$, for every $n$. For such a group, by results of
Bestvina-Mess \cite{BeMe}, the (topological) dimension of its
boundary is $n-1$ and thus, by \cite{BuLe}, the asymptotic
dimension of the group is $n$.

\subsection{A construction of high dimensional AHA groups}
\label{chcg}

Let $X_0$ be a finite $5$--large simplicial complex
with $SD_2^{\ast}$ links such that $H^{n_0}(X_0;\mathbb Q)\neq 0$.
We will construct inductively, out of $X_0$,
a sequence $(X_i)_{i\geqslant0}$ of
finite $5$--large simplicial complexes with $SD_2^{\ast}$ links such that
$H^{n_0+i}(X_i;\mathbb Q)\neq 0$, for all $i$.
\medskip

Assume that $X=X_k$ is constructed. We describe the construction of $X'=X_{k+1}$.
Let $(W,S)$ be the right-angled Coxeter system whose nerve is $X$,
i.e.\ $X=L(W,S)$ (see \cite{Davis} for details on Coxeter groups -- we follow here the Davis' notation).
The Davis complex $\Sigma=\Sigma(W,S)$ of $(W,S)$ is a CAT(0) cubical complex in which the link of every vertex is $X$ (so that in fact, in our case, $\Sigma$ is a CAT(-1) cubical complex). The thickening $Th(\Sigma)$ of $\Sigma$ (see Definition~\ref{thick}) is weakly systolic, by Proposition~\ref{p:thcat-1}.

Taking into account the assumption about $SD_2^{\ast}$ property of links, we obtain the following version of Main Theorem from \cite{O-chcg}, which makes the whole construction working.

\begin{tw}
 \label{indCox}
 Let $X$ be a finite $5$--large simplicial complex with $SD_2^{\ast}$ links, and such that $H^{n_0+k}(X;\mathbb Q)\neq 0$. Then there exists a torsion-free finite index subgroup $H$ of the right-angled Coxeter group $(W,S)$ with nerve $X$, having the following properties.
The quotient $X':=H\backslash Th(\Sigma(W,S))$ is a finite $5$--large simplicial complex with $SD_2^{\ast}$ links, satisfying the condition $H^{n_0+k+1}(X';\mathbb Q)\neq 0$.
\end{tw}
\begin{proof}
Since Coxeter groups are virtually torsion-free (see \cite[Corollary D.1.4]{Davis})
and residually finite (see \cite[Section 14.1]{Davis}), there exists a torsion-free finite index subgroup $H<W$
with the following property. The minimal displacement of the action of $H$ on $Th(\Sigma)$
(induced by the action of $W$ on $\Sigma$ ) is at least $5$.
For such a subgroup $H$, by Proposition~\ref{p:quot}, the quotient $X':=H\backslash Th(\Sigma)$ is a $5$--large complex with $SD_2^{\ast}$ links

The fact that $H^{n_0+k+1}(X';\mathbb Q)\neq 0$ follows from the more general
result proved in \cite{O-chcg}
(namely Theorem 4.6 from that paper). \end{proof}

\begin{cor}
  \label{corCox}
The right-angled Coxeter group $W'$ with nerve $X'$ is AHA of virtual cohomological dimension at least $n_0+k+2$.
	\end{cor}
  \begin{proof}
    Since $H^{n_0+k+1}(X';\mathbb Q)\neq 0$ we have that $H^{n_0+k+1}(X';\mathbb Z)\neq 0$, and thus by \cite[Corollary 8.5.5]{Davis}, the virtual cohomological dimension of $W'$ is at least $n_0+k+2$.
    Observe that $W'$ acts geometrically on its Davis complex $\Sigma'$, being a locally $5$--large cubical complex with links of all vertices isomorphic to $X'$. Thus, by Proposition~\ref{p:thcat-1}, the thickening $Th(\Sigma')$ is weakly systolic, and by Lemma~\ref{sd2th}, it has $SD_2^{\ast}$ links. As a group acting geometrically on a weakly systolic complex with $SD_2^{\ast}$ links, the group $W'$ is AHA, by Theorem~\ref{systaha}.
  \end{proof}

Now, if we started with $X_0$ being a complex with $SD_2^{\ast}$ links (e.g.\ with $X_0$ being a $5$--large graph), then, after performing $k$ steps of the above construction we get a corresponding complex $X_k$ and the associated AHA group of virtual cohomological dimension at least $k+1$.

\medskip\noindent
{\bf Remark.} If we start with a $6$--large complex $X_0$ and then, at every (inductive) step, we
assure that the corresponding injectivity radius is at least $6$, then all the groups we obtain
will be systolic and thus also AHA -- compare \cite[Subsection 6.1]{O-chcg}.

\section{Finiteness properties of AHA groups}
\label{FAHA}
In this section we prove that finitely presented AHA groups have type $F_{\infty}$.
Besides being of its own interest this result is used in Section \ref{CIAHA}
when we study connectedness at infinity.

\begin{tw}
\label{AHAF}
A finitely presented AHA group has type $F_{\infty}$.
\end{tw}
\begin{proof}
Let $G$ be a finitely presented AHA group. Consider the family of
its Rips complexes (with respect to a fixed generating set) $\lk P_d(G) \rk_{d=D}^{\infty}$
with $D$ being a natural number given by Lemma \ref{scRc},
i.e.\ such that every $P_d(G)$ is simply connected.
Then $P_D(G) \subseteq P_{D+1}(G) \subseteq \cdots$ is a $G$--filtration
of the contractible $G$--CW complex
$P_{\infty}(G)=\bigcup_{d=D}^{\infty}P_d(G)$ -- see \cite[Lemma 7]{Al2}.
The group $G$ acts on $P_{\infty}(G)$ with finite (and thus having type
$F_{\infty}$) stabilizers and $G\backslash P_d(G)$ is finite for every $d$.

Since every $P_d(G)$ is simply connected and $G$ is AHA it follows immediately
that the filtration $\lk P_d(G) \rk$ is essentially $n$--connected, for every $n$.
Thus, by the Brown's Finiteness Criterion (Theorem \ref{BFC}),
the group $G$ has type $F_n$ for every $n$ and hence it has type $F_{\infty}$.
\end{proof}

\noindent
{\bf Remarks.} 1) The finite presentation assumption is essential.
For a finite group $F$, the lamplighter group
$F \wr \mathbb Z =F^{\mathbb Z}\rtimes \mathbb Z$
(where $\mathbb Z$ acts by shifts on
$F^{\mathbb Z}$) is a group without finite presentation,
and thus not of type $F_\infty$.
By a result of J.\ Zubik \cite{Z}, this group is AHA
since its asymptotic dimension is $1$.

2) A construction of infinitely presented graphical small cancellation groups of infinite asymptotic dimension is presented in \cite{O-sc}. 
By Lemma~\ref{prescom'} (see Remarks following it and Example~\ref{4.4}(3)) such groups are AHA. We do not know if there exist 
finitely presented AHA groups of infinite asymptotic dimension. Note, that in case of infinite dimension (as e.g.\ for monster groups constructed in \cite{O-sc}) the AHA property becomes an important tool, e.g.\ for restricting the class of subgroups.  

3) Assume that a finitely presented group $G$ satisfies the following property,
which is weaker than AHA. For every $r>0$ there exists $R>0$ such that
every subset of $G$ is $(i;r,R)$--aspherical in itself, for $i=2,\ldots,n-1$.
Then, as in the proof of Theorem \ref{AHAF}, we get that $G$ has type $F_n$
and a version of Corollary \ref{FpiAHA} holds.

\section{Connectedness at infinity of AHA groups}
\label{CIAHA}
In this section we study connectedness at infinity of AHA groups and provide some applications to the case of (weakly) systolic groups.

In general it is not obvious how to define \emph{homotopy groups at infinity} -- see \cite[Chapter 16]{Geo}.
However the following definition of vanishing of such groups seems natural.

\begin{de}[$\pi_i^{\infty}=0$]
\label{nasf}
For a topological space $X$ we say that $\pi_i^{\infty}(X)=0$, i.e.\ the \emph{$i$--th homotopy group at infinity vanishes} whenever the following holds.
For every compact $K\subseteq X$ there exists a compact $L\subseteq X$ with $K\subseteq L$ such that every map $S^i=\partial B^{i+1}\to X\setminus L$ extends to a map $B^{i+1}\to X\setminus K$.
\end{de}

\rem
In \cite{O-ciscg} the condition $\pi_i^{\infty}(X)=0$ is denoted by ``$\mr{Conn}_i^{\infty}(X)$''. We think the current notation is more natural. Moreover, it was used sometimes in the literature.
\medskip

Recall that a metric space $(X,d)$ is \emph{uniformly $n$--connected}, $n\geqslant 0$, if for every $0\leqslant i\leqslant n$ and for every $r>0$ there exists $R=R(r)>0$ such that the following property holds. Every map $S^i=\partial B^{i+1}\to N_r(\{x\})$ (i.e.\ with the image contained in some $r$--ball in $X$) can be extended to a map $B^{i+1}\to N_R(\{x\})$. Observe that if a group acts geometrically on an $n$--connected complex
then this complex is uniformly $n$--connected.

The next theorem is crucial for the definition of vanishing of the homotopy groups at infinity for groups instead of spaces. It is, together with its proof presented here, a folklore result.
That way of proving it is often referred to as ``connecting the dots method''. Because of lack of a precise reference we include the proof here. For similar approaches see e.g.\ \cite{FiW}, \cite{FO},
\cite[Chapter 18.2]{Geo}. Another approach can be found in \cite[Chapter 17.2]{Geo} and \cite[Proposition 2.10]{O-ciscg}.

\begin{tw}[$\pi_i^{\infty}=0$ is u.e.\ invariant]
\label{qipi}
Let $(X,d)$ and $(X',d')$ be two uniformly $n$--connected proper metric spaces. Assume that $X$ and $X'$ are uniformly equivalent. Then
$\pi_i^{\infty}(X)=0$ iff $\pi_i^{\infty}(X')=0$, for every $i\leqslant n$.
\end{tw}

\begin{proof}
It follows directly from Proposition \ref{aspi=pi} and Proposition \ref{asinv}.

For the reader's convenience let us sketch the proof here. It will also give a rough idea on the proofs of the two propositions mentioned above presented in Section \ref{AsAHA}.

Assuming $\pi_i^{\infty}(X')=0$ we want to show $\pi_i^{\infty}(X)=0$.
To do this, for every compact $K\subseteq X$ we have to find a compact
$L\supset K$ such that the map $\pi_i(X\setminus L)\to \pi_i(X\setminus K)$ induced by the inclusion is trivial. Let $h\colon X \to X'$ be a u.e.
By $\pi_i^{\infty}(X')=0$,
there exists a compact $L'\supset h(K)$ such that the map $\pi_i(X'\setminus L')\to \pi_i(X'\setminus h(K))$ is trivial. We claim that, for some constant $a>0$ depending only on $X,X'$, a compact $L=N_a(h'(L'))$ is as desired, where $h'$ is a u.e.\ coarsely inverse to $h$.

To prove the claim let $f\colon S^i\to X\setminus L$ be given.
We can find a triangulation $S$ of $S^i$ such that the image by $f$ of every simplex of $S$ is small (i.e.\ has the diameter smaller than a given constant).
Then we can find (if $a$ is big enough) an extension $f'\colon S^i \to X'\setminus L'$ of $h\circ f|_{ S^{(0)}}$ with small images of simplices. It can be done by extending
$h\circ f|_{ S^{(0)}}$ over all simplices in a controlled way, using the uniform
$n$--connectedness of $X'$ (more precisely: using Lemma \ref{ucfill}). By the choice of $L'$ there exists an extension $F'\colon B^{i+1} \to X'\setminus K'$ of
$f'$. Then we consider the map $h'\circ F'|_{B^{(0)}}\colon B^{(0)}\to X$, where $B$ is a triangulation of $B^{i+1}$ with $\partial B=S$ and such that the images of simplices of $B$ are small. Again, by the uniform $n$--connectedness of $X$ (Lemma \ref{ucfill}), we can extend $h'\circ F'|_{B^{(0)}}$ to a map $F\colon B^{i+1}\to X$ in such a way that images of simplices of $B$ are small. If $a$ is sufficiently large it follows that
$\mr{Im}\, F\subseteq X\setminus K$ and that $F|_{S^i}$ is homotopic in $X\setminus K$ with $f$. Thus the lemma is proved.
\end{proof}

In view of the theorem above the following definition makes sense.

\begin{de}[$\pi_i^{\infty}(G)=0$]
\label{gasf}
For a group $G$ we say that its \emph{$i$--th homotopy group at infinity vanishes}, denoted $\pi_i^{\infty}(G)=0$ if for some (and hence for any)
$i$--connected metric space $(X,d)$, on which $G$ acts geometrically, we have
$\pi_i^{\infty}(X)=0$.
\end{de}

\begin{tw}[$\pi_i^{\infty}=0$ for AHA]
\label{piAHA}
Let $G$ be an AHA group acting geometrically on an $n$--connected metric space. Then $\pi_i^{\infty}(G)=0$, for every $2\leqslant i\leqslant n$.
\end{tw}
\begin{proof}
This follows directly from Proposition \ref{AHAaspi0} and Proposition \ref{aspi=pi} in Section \ref{AsAHA}.
\end{proof}

By the finiteness properties of AHA groups (Theorem \ref{AHAF}) we get immediately the following.

\begin{cor}
\label{FpiAHA}
A finitely presented AHA group $G$ satisfies $\pi_i^{\infty}(G)=0$, for every $i\geqslant 2$.
\end{cor}

Having established the asphericity at infinity of AHA groups we would like to proceed as in \cite{O-ciscg}, i.e.\ to show that (finitely presented) AHA groups are not simply connected at infinity. This needs however further assumptions on the finiteness of a dimension of a group.
One version of a theorem concerning that question is the following.

\begin{tw}[$\pi_1^{\infty}\neq 0$ for AHA]
\label{npi1}
Let $G$ be a one-ended finitely presented AHA group such that
either

a) $\mr{vcd} (G) <\infty$ or

b) there exists a finite dimensional $G$--CW complex $X$ such that $G\backslash X$ is compact and cell stabilizers are finite.

Then $G$ is \emph{not simply connected at infinity}, i.e.\ it is not true that $\pi_1^{\infty}(G)=0$.

\end{tw}
\begin{proof}
We treat simultaneously the two cases. Note that without loss of generality we may assume that $G$ is torsion-free in case a).
Assume, by contradiction, that $\pi_1^{\infty}(G)=0$. Then, by Corollary \ref{FpiAHA} we have that $\pi_i^{\infty}(G)=0$ for every $i\geqslant 0$.
By the Proper Hurewicz Theorem
\cite[Theorem 17.1.6]{Geo}
we have that $G$ is $i$--acyclic at infinity with respect
to ${\mathbb Z}$, for arbitrary $i$.
By \cite[Corollary 4.2]{GeoM1} (compare also \cite[Corollary]{GeoM2} and \cite[Theorem 13.3.3]{Geo})
it follows that $H^i(G;\mathbb ZG)=0$ for all $i$.

But this contradicts \cite[Proposition 13.10.1]{Geo} in case a) or \cite[Proposition 2.9]{O-ciscg} in case b).
\end{proof}

\rems (1) The finite dimensionality assumptions are essential for the proof given above. The Thompson's group is an $F_{\infty}$ group satisfying
$\pi_i^{\infty}(G)=0$, for every $i\geqslant 0$. But its cohomological dimension is infinite -- see \cite[Section 13.11]{Geo}. We do not know whether there exist finitely presented AHA groups of infinite dimension.

(2) Observe that systolic groups and groups acting geometrically on
those weakly systolic complexes whose all full subcomplexes
are aspherical -- as described in Section \ref{weak} -- satisfy the assumption b) of Theorem \ref{npi1}.
\medskip

In the remaining part of the section we present sample corollaries of the results concerning connectedness at infinity given above.
In what follows we assume that $\mathbb R^n$ is equipped with a proper metric consistent with the topology, and
groups act by isometries with respect to the metric.
The first result is an extension of \cite[Corollary 6.3]{JS2}.

\begin{cor}
\label{nman}
Let $n\geqslant 3$. Groups acting geometrically on $\mathbb R^n$ are not AHA. In particular,
the fundamental group of a closed manifold covered by $\mathbb R^n$ is not AHA.
\end{cor}

As an immediate consequence we obtain the following corollary that extends results obtained in \cite[Section 3.1]{O-ciscg},
where only the case of torsion-free systolic groups was considered. It follows directly from Corollary~\ref{nman} in view
of the fact that finitely generated subgroups of AHA groups are AHA themselves, and systolic groups are AHA.

\begin{cor}
\label{nmans}
Let $n\geqslant 3$. Groups acting geometrically on $\mathbb R^n$ are
not isomorphic to subgroups of systolic groups. In particular,
the fundamental group of a closed manifold covered by $\mathbb R^n$ is not isomorphic to a subgroup of a systolic group.
\end{cor}

\rem Other examples of groups that are not AHA can be found in \cite[Section 3.1]{O-ciscg}.
\medskip

The following corollary is a slight extension of \cite[Corollary 3.3]{O-ciscg} in the case of systolic groups.

\begin{cor}
\label{ext}
Let $K \rightarrowtail G \twoheadrightarrow H$ be a short exact sequence of infinite finitely generated groups with $G$ being a finitely presented AHA group satisfying the assumption a) or b) of Theorem \ref{npi1}.
Then none of the following conditions hold:

a) There is a one-ended finitely presented normal subgroup $N$ of $K$ and $G/K$ or $K/N$ is infinite;

b) $H$ is one-ended and $K$ is contained in a finitely presented infinite index subgroup of $G$.
\end{cor}
\begin{proof}
In view of Theorem \ref{npi1}, the corollary follows directly from \cite[Main Theorem]{P} in case a) and from
\cite[Theorem 2]{M2} in case b).
\end{proof}

The result above implies in particular the following. Let
$K \rightarrowtail G \twoheadrightarrow H$ be a short exact sequence of infinite finitely presented groups with $G$ being a finitely presented AHA group satisfying the assumption a) or b) of Theorem \ref{npi1}. Then neither $K$ nor $H$ has one end.

\section{Asymptotic connectedness at infinity}
\label{AsAHA}

In this section we introduce a new notion: vanishing of asymptotic homotopy groups at infinity. It is an asymptotic counterpart of the topological notions considered in the previous section. The first reason for introducing this is that it appears naturally in the studies of connectedness at infinity of AHA groups -- see Proposition~\ref{AHAaspi0} and its proof. On the other hand the asymptotic notion seems to be a valuable extension of its topological analogue. Usually, one defines the condition $\pi_i^{\infty}(G)=0$ only when $G$ has type $F_i$ -- cf.\ Definition \ref{gasf} and \cite[Section 17.2]{Geo} -- whereas the asymptotic counterpart is useful for finitely generated groups (compare Examples after Proposition~\ref{AHAaspi0} below).

\begin{de}[$\mbox{as--}\pi_i^{\infty}=0$]
\label{aspix}
Let $(X,d)$ be a metric space. We say that $\mbox{as--}\pi_i^{\infty}(X)=0$, i.e.\ that the \emph{$i$--th asymptotic homotopy group at infinity of $X$ vanishes} if the following holds.
For every $r>0$ there exists $R=R(r)>0$ such that for every bounded set $K\subseteq X$ there exists a bounded set $L\subseteq X$ containing $K$, for which $X\setminus L$ is $(i;r,R)$--aspherical in $X\setminus K$.
\end{de}

As a direct consequence of the definition above and the definition of AHA (Definition 2.1(b)) we have the following.

\begin{prop}
\label{AHAas0}
Let $(X,d)$ be an AHA metric space. Then $\mbox{as--}\pi_i^{\infty}(X)=0$, for every $i\geqslant 2$.
\end{prop}

Next proposition establishes coarse invariance of the just
introduced notion, which allows to extend it to finitely generated
groups.

\begin{prop}[$\mbox{as--}\pi_i^{\infty}=0$ is u.e.\ invariant]
\label{asinv}
Let $(X,d)$ and $(X',d')$ be metric spaces and let $h\colon X\to X'$ be a $(g_1,g_2,N)$--uniform equivalence between them. Then
$\mbox{as--}\pi_i^{\infty}(X)=0$ iff $\mbox{as--}\pi_i^{\infty}(X')=0$, for every $i\geqslant 0$.
\end{prop}
\begin{proof}
The idea of the proof was explained in the (sketch of the) proof of Theorem
\ref{qipi}. It is rather straightforward.

Let $h'\colon X' \to X$ be a $(g_1',g_2',N')$--uniform equivalence that is coarsely inverse to $h$. Let $d(x,h'\circ h(x)),d(y,h\circ h'(y))\leqslant M$, for every $x\in X$ and $y\in Y$.
Without loss of generality we can assume that $g_1'=g_1\leqslant g_2=g_2'$, that they are monotone
and that $N=N'=M$.
Assume that $\mbox{as--}\pi_i^{\infty}(X')=0$. We will prove that then
$\mbox{as--}\pi_i^{\infty}(X)=0$. The other implication follows when one
interchanges $h$ and $h'$.

Let $r>0$. We claim that then there exists a constant $R=R(r)$ as desired in the definition
 of $\mbox{as--}\pi_i^{\infty}(X)=0$.

We prove the claim. Let $K\subseteq X$ be bounded and let $a>0$ be such that $g_1(a)>N$.
Then $h(K)\subseteq X'$ is bounded and there exists a bounded $L'\subseteq X'$ such that $X'\setminus L'$ is $(i;g_2(r),R')$--aspherical in $X'\setminus K'=X'\setminus N_a(h(K))$, where $R'$ is the constant from the definition of $\mbox{as--}\pi_i^{\infty}(X)=0$ corresponding to $g_2(r)$.
We show that $X\setminus L=X\setminus N_a(h'(L'))$ (which is obviously cobounded) is $(i;r,R)$--aspherical in $X\setminus K$, where $R=R(r)=N + g_2(R')$ is claimed to be the desired constant.

Let $f\colon S\to P_r(X\setminus L)$ be a map from a triangulated $i$--sphere $S$. The map $h$ induces a simplicial map $h_r\colon
P_r(X)\to P_{g_2(r)}(X')$. Consider the map $h_r \circ f \colon S\to P_{g_2(r)}(X')$. We show that $\mr{Im}(h_r\circ f) \subseteq X'\setminus L'$.
To show this let $x\in X \setminus L$ and $x'\in L'$ be two arbitrary points. It is enough to prove that $h(x)\neq x'$.
We have
\begin{align*}
d'(h(x),x')& \geqslant d'(h(x),h\circ h'(x')) - d'(h\circ h'(x'),x') \\
& \geqslant d'(h(x),h\circ h'(x')) - N  \geqslant g_1(d(x,h'(x'))) - N \\
& \geqslant g_1(a)-N >0.
\end{align*}
Thus $h(x)\neq x'$ and it follows that $\mr{Im}(h_r\circ f) \subseteq X'\setminus L'$.
By the choice of $L'$ there exists an extension $F'\colon B \to P_{R'}(X'\setminus K')$ of $h_r\circ f$, where $B$ is a triangulation of $B^{i+1}$ with $\partial B=S$. Without loss of generality we may assume that $S$ is full in $B$ (subdividing $B$ if necessary).

We define the desired extension $F$
of $f$ as follows. For a vertex $v\in S= \partial B$ we set $F(v)=f(v)$ and for a vertex $w$ outside $S$ we set $F(w)=h'\circ F' (w)$.
We show that this extends to a simplicial map
$F\colon B \to P_{R}(X)$. Let $v,w\in B$ be two vertices joined by an edge in $B$.
If $v,w\in S$ then (by the fact that $S$ is full) $d(f(v),f(w))\leqslant r\leqslant R$.
If $v\in S$ and $w$ not is in $S$ we have the following
\begin{align*}
d(F(v),F(w)) & = d(f(v),F(w)) \\
& \leqslant d(f(v),h'\circ h (f(v))) +
d(h'\circ h (f(v)), F(w)) \\
& \leqslant N + d(h'\circ F'(v), h'\circ F'(w))\leqslant N+g_2(d'(F'(v),F'(w))) \\
& \leqslant N + g_2(R')=R.
\end{align*}
If $v$ and $w$ lie outside $S$ we have
\begin{align*}
d(F(v),F(w)) & \leqslant d(h'\circ F'(v), h'\circ F'(w)) \leqslant g_2(d'(F'(v),F'(w))) \\
& \leqslant g_2(R') \leqslant R.
\end{align*}
So that in every case $F(v)$ and $F(w)$ are joined in $P_R(X)$ by an edge and thus we get a simplicial map $F\colon B \to P_R(X)$ extending $f$.

We show now that $\mr{Im}\; F \subseteq X\setminus K$. Let $x'\in X'\setminus K'$ and $x\in X\setminus K$ be two arbitrary points. It is enough to prove that $x\neq h'(x')$.
We have
\begin{align*}
d(x,h'(x'))& \geqslant d(h'\circ h(x),h'(x')) - d'(x, h'\circ h(x)) \\
& \geqslant d(h'\circ h(x),h'(x')) - N  \geqslant g_1(d(h(x),x')) - N \\
& \geqslant g_1(a)-N >0.
\end{align*}
Thus $x\neq h'(x')$ and it follows that $\mr{Im}\; F \subseteq X\setminus K$.

Summarizing, for an arbitrarily chosen $f\colon S \to P_r(X\setminus L)$ we constructed its extension $F\colon B\to P_R(X\setminus K)$. Hence
$X\setminus L$ is $(i;r,R)$--aspherical in $X\setminus K$ and thus
$\mbox{as--}\pi_i^{\infty}(X)=0$.
\end{proof}

In view of the result above the following definition makes sense.
\begin{de}[$\mbox{as--}\pi_i^{\infty}(G)=0$]
\label{aspiG}
Let $G$ be a finitely generated group. We say that \emph{$\mbox{as--}\pi_i^{\infty}(G)=0$}, i.e.\ that
the \emph{$i$--th asymptotic homotopy group at infinity of $G$ vanishes} if
$\mbox{as--}\pi_i^{\infty}(G)=0$ for $G$ seen as a metric space with the word metric induced by some (and thus by any) finite generating set.
\end{de}

The next proposition follows directly from Proposition \ref{AHAas0}.

\begin{prop}
\label{AHAaspi0}
If $G$ is an AHA group then $\mbox{as--}\pi_i^{\infty}(G)=0$ for every $i\geqslant 2$.
\end{prop}

\noindent
{\bf Examples.} For $i\geqslant 2$ we have $\mbox{as--}\pi_i^{\infty}(\mathbb Z_2 \wr \mathbb Z)=0$, since the lamplighter group is AHA -- see the remark just after Theorem \ref{AHAF}. On the other hand this group is not finitely presented and thus one cannot (a priori) say what does ``$\pi_i^{\infty}(\mathbb Z_2 \wr \mathbb Z)=0$'' mean.
Similarly, by Remark (2) after Lemma~\ref{prescom'}, we have that $\mbox{as--}\pi_i^{\infty}(G)=0$ for infinitely presented small cancellation groups $G$.
This shows that the asymptotic asphericity (or connectedness) at infinity can be sometimes more useful than the topological counterpart.
\medskip

Now we show that in some cases (that occur naturally in the context of group actions), vanishing of the ``usual'' homotopy groups at infinity is equivalent to vanishing of the asymptotic ones. First we prove a simple technical lemma.

\begin{lem}
\label{ucfill}
Let $(X,d)$ be a uniformly $n$--connected ($n<\infty$) metric space and let $r>0$.
Then there exists $C=C(r)>0$ with the following property.
Let $Y$ be a simplicial complex of dimension at most $n$ and let
$f\colon Y'\to X$ be a continuous map from a subcomplex $Y'\subseteq Y$ containing all vertices of $Y$.
Assume that for every vertices $v,w$ joined by an edge of $Y$ we have $d(f(v),f(w))\leqslant r$
and that $f(\sigma)\subseteq N_r(\{f(v)\})$ for every simplex $\sigma$ of $Y'$ and every its vertex $v$.
Then there exists an extension $F\colon Y\to X$ of $f$ such that $F(\sigma)\subseteq N_{C}(\{v\})$, for every simplex $\sigma$ of $Y$ and every vertex $v$ of $\sigma$.
\end{lem}
\begin{proof}
Let $f,Y$ be as in the assumptions.
We construct an extension $F_1\colon Y'\cup Y^{(1)} \to X$ of $f$ as follows. For every pair of vertices $v,w\in Y$
connected by an edge of $Y$
we choose, by uniform $n$--connectedness (with the constant $R(r)$), a path $Y^{(1)}\ni [v,w]\to
N_{R(r)}(\{f(v)\})\cap N_{R(r)}(\{f(w)\})$ with endpoints $f(v),f(w)$.
Then we proceed by induction. Assume we have found a constant $R_k>0$ and a map $F_k\colon Y'\cup Y^{(k)} \to X$ with
$F_k(\sigma)\subseteq N_{R_k}(\{f(v)\})$, for every $k$--simplex in $Y$ and every its
vertex $v$.
(Note that for $k=1$ we have done this above, with $R_1=R(r)$.)
Then, by uniform $n$--connectedness, we can find $R_{k+1}>R_k$ and an extension $F_{k+1}\colon Y'\cup Y^{(k+1)} \to X$ with $F_{k+1}(\tau)\subseteq N_{R_{k+1}}(\{f(w)\})$, for every $(k+1)$--simplex in $Y$ and every its
vertex $w$.

Eventually we obtain $F=F_n$ and $C=R_n$ satisfying the hypotheses of the lemma.
\end{proof}

\begin{prop}
\label{aspi=pi}
Let $(X,d)$ be a uniformly $n$--connected proper metric space. Then
$\mbox{as--}\pi_i^{\infty}(X)=0\Longleftrightarrow \pi_i^{\infty}(X)=0$, for $i\leqslant n$.
\end{prop}
\begin{proof}
{\bf ($\Rightarrow$) } Let $K\subseteq X$ be compact and thus bounded. We prove that there exists a compact $L\supseteq K$ such that the map $\pi_i(X\setminus L)\to \pi_i(X\setminus K)$ induced by the inclusion is trivial.

Let $R>r>0$ be the constants from the definition of $\mbox{as--}\pi_i^{\infty}(X)=0$, i.e.\ $r,R$ are such that for every bounded $K'\subseteq X$, there is a bounded $L'\subseteq X$ such that $X\setminus L'$ is $(i;r,R)$--aspherical in $X\setminus K'$.

Let $L\subseteq X$ be a bounded set such that $X\setminus L$ is $(i;r,R)$--aspherical in $X\setminus N_{C}(K)$, where $C=C(R)$ is the constant from Lemma \ref{ucfill}. We claim that $L$ is as desired.

Let $f\colon S^i \to X\setminus L$ be a map from the $i$--sphere $S^i$.
We have to show that it has an extension $F\colon B^{i+1}\to X\setminus K$.
Take a triangulation $S$ of $S^i$ such that for
every simplex $\sigma$ of $S$ we have $f(\sigma)\subseteq N_r(\{v\})$, for each vertex $v$ of $\sigma$. Then we have an induced (by $f$) map
$f_r\colon S \to P_r(X\setminus L)$. By $\mbox{as--}\pi_i^{\infty}(X)=0$ this map has a simplicial extension $F_R\colon B\to P_R(X\setminus N_{C}(K))$, where $B$ is a triangulation of $B^{i+1}$ with $\partial B=S$.
Then the union of maps $F'=F_R \cup f \colon B^{(0)}\cup S \to X\setminus N_{C}(K)$ has the property that $F'(\sigma)\subseteq N_R(\{v\})$ for every simplex $\sigma$ of $B$ and every its vertex $v$. Thus, by Lemma \ref{ucfill}, there is an extension $F\colon B^{i+1} \to X$ such that $F(\sigma)\subseteq N_{C}(\{v\})$ for every simplex $\sigma \in B$ and every its vertex $v$.
Thus $\mr {Im}\, F\subseteq X\setminus K$ and we proved that
$\pi_i^{\infty}(X)=0$.
\vspace{0.3cm}

\noindent
{\bf ($\Leftarrow$) } Let $r>0$. We show that for every bounded $K\subseteq X$ there exists a bounded $L\subseteq X$ containing $K$ such that $X\setminus L$
is $(i;r,2r)$--aspherical in $X\setminus K$. Let such $K$ be given.
By $\pi_i^{\infty}(X)=0$ there exists a bounded $L'\supset K$, for which the
induced (by inclusion) map $\pi_i(X\setminus L')\to \pi_i(X\setminus K)$ is trivial. We claim that $L=N_{C}(L')$ is as desired, where $C=C(r)$ is the constant from Lemma \ref{ucfill}.
Let $f\colon S \to P_r(X\setminus L)$ be a simplicial map from a triangulation $S$ of an $i$--sphere $S^i$. By Lemma \ref{ucfill} there exists an extension $f'\colon S^i \to X$ of $f|_{S^{(0)}}$ with the image contained in $X\setminus L'$ (compare the proof of the opposite implication).
By $\pi_i^{\infty}(X)=0$ there exists an extension $F' \colon B^{i+1} \to X\setminus K$ of $f'$. Then we can find a triangulation $B$ of $B^{i+1}$ such that $\partial B=S$ and such that the simplicial map induced by $F'$ is a map
$F\colon B \to P_{2r}(X\setminus K)$ extending $f$.
\end{proof}

\section{No $2$--disk in the boundary}
\label{bdry}

This section is devoted to the proof of the following results.

\begin{tw}[No $2$--disk in CAT(0) and Gromov boundaries]
\label{d1}
Let $X$ be a geodesic metric space which is AHA. Then
\begin{enumerate}
\item if $X$ is proper and $\delta$--hyperbolic then its Gromov boundary
contains no $2$--disk;
\item if $X$ is complete and CAT(0) then its boundary (equipped with cone topology)
contains no $2$--disk.
\end{enumerate}
\end{tw}

\begin{tw}[No $2$--disk in systolic boundaries]
\label{d2}
Let $X$ be a systolic simplicial complex. Then its systolic boundary
(as defined in \cite{OP}) contains no $2$--disk.
\end{tw}

The proof of Theorem \ref{d1} occupies Subsection \ref{pr9.1},
while that of Theorem \ref{d2} is preceded by a preparatory Subsection \ref{gap},
and occupies Subsection \ref{sysgap}.

\medskip\noindent
{\bf Remarks.}
(1) The fact that the boundary contains no $2$--disk was established earlier
by the second author, in \cite{Sw-propi}, for the class of $7$--systolic simplicial complexes
(see \cite{Sw-propi}, Main Theorem and Remarks 2.2(1) and 2.5(3)).
Since $7$--systolic complexes are both systolic and $\delta$--hyperbolic,
and since their Gromov boundaries coincide with systolic ones, both
Theorems \ref{d1} and \ref{d2} generalize the result from \cite{Sw-propi}.
\noindent

(2) Theorems \ref{d1} and \ref{d2} are not trivially true only when
the corresponding boundaries have topological dimension at least $2$.
Examples of AHA groups with arbitrarily large dimension of the boundary,
including CAT(0), word-hyperbolic and systolic groups, are constructed
e.g.\ in \cite{JS1} and in Subsection~\ref{chcg} of the current paper.

\subsection{Proof of Theorem \ref{d1}.}
\label{pr9.1}

All the results above follow from a single but more technical observation,
Proposition \ref{g3}, in view of the rather known properties
of the corresponding classes of spaces $X$.
A special case of Proposition \ref{g3}, sufficient to prove Theorem \ref{d1}, is
Proposition \ref{d4} below.
To formulate this result, we need some preparations.

Recall that a {\it geodesic ray} in a metric space $X$ is an isometric map
$r\colon \mathbb R_+\to X$; it is {\it based at a point $x_0\in X$} if $r(0)=x_0$.

\begin{de}[$D^2$--conical]
\label{d3}
Let $r_\lambda:\lambda\in\Lambda$ be a family of geodesic
rays in a metric space $(X,d_X)$, based at one point $x_0\in X$, parametrized by
a metric space $(\Lambda,d_\Lambda)$. We say that this family is \emph{conical}
(with respect to the metric $d_\Lambda$ on the parameter space)
if for some constant $C>0$ the following two conditions are satisfied:

\begin{enumerate}
\item $\forall N>0\;\exists\varepsilon>0\; \forall t\leqslant N$
if $d_\Lambda(\lambda,\mu)\leqslant \varepsilon$ then $d_X(r_\lambda(t),r_\mu(t))\leqslant C$,
\item $\forall L\geqslant C\;\forall\varepsilon'>0\;\exists N'>0\;\forall t\geqslant N'$ if
$d_X(r_\lambda(t),r_\mu(t))\leqslant L$ then $d_\Lambda(\lambda,\mu)\leqslant \varepsilon'$.
\end{enumerate}

A metric space is \emph{$D^2$--conical} if it admits a conical family
of geodesic rays with the parameter space $\Lambda$ homeomorphic
to the $2$--disk $D^2$.
\end{de}

The next proposition reveals the relationship between $D^2$--conicality and AHA.

\begin{prop}
\label{d4}
If a metric space is $D^2$--conical then it is not AHA.
\end{prop}

Before proving the proposition, we explain how it implies Theorems \ref{d1}.
We start the explanation with
two results which exhibit that the two classes of metric spaces $X$ occurring
in Theorem \ref{d1} admit conical
families of geodesic rays parameterized by the ideal boundaries $\partial X$
(equipped with appropriate metrics).
The first of these results, Lemma \ref{d5}, refers to the well known properties of {\it visual metrics}
on Gromov boundaries of $\delta$--hyperbolic spaces
(see \cite[Chapter III.H, Proposition 3.21, p.\ 435]{BrHa}).
Recall that in a proper geodesic $\delta$--hyperbolic metric space $X$,
for any $x_0\in X$ and any $\xi\in\partial X$ there is a geodesic ray
based at $x_0$ and diverging to $\xi$ (\cite[Chapter III.H, Lemma 3.1, p.\ 427]{BrHa}).

\begin{lem}[Gromov boundary is conical]
\label{d5}
Let $X$ be a proper geodesic $\delta$--hyperbolic metric space,
and let $d_{\partial X}$ be a visual metric on the Gromov boundary $\partial X$,
with respect to a point $x_0\in X$. For each $\xi\in\partial X$ choose any
geodesic ray $r_\xi$ in $X$ based at $x_0$ and diverging to $\xi$. Then the family
$r_\xi:\xi\in\partial X$ is conical (with the parameter space
$\Lambda=(\partial X,d_{\partial X})$) for the constant $C=2\delta$.
\end{lem}

The next result exhibits the existence of less known natural metrics
(compatible with the cone topology) on boundaries $\partial X$
of complete CAT(0) spaces $X$. Recall that in a complete CAT(0) space
any point $x_0\in X$ can be connected to any $\xi\in\partial X$
with the unique geodesic ray (see \cite[Chapter II, Proposition 8.19]{BrHa}).
Recall also that for distinct geodesic rays $r,r'$ based at a common point of $X$
the function $t\to d_X(r(t),r'(t))$ is convex (see \cite[Chapter II, Proposition 2.2]{BrHa}),
and hence for any real $A>0$ there is exactly one $t$ with $d_X(r(t),r'(t))=A$.

\begin{prop}[Conical metric on CAT(0) boundary]
\label{d6}
Let $X$ be a complete CAT(0) space, $x_0\in X$ a point,
and $A>0$ a real number. For $\xi,\eta\in\partial X, \xi\ne\eta$, let
$r_\xi,r_\eta$ be the unique geodesic rays from $x_0$ to $\xi$ and $\eta$, respectively.
Put $d_A(\xi,\eta):=t^{-1}$, where $t$ is the unique number with $d_X(r_\xi(t),r_\eta(t))=A$.
\begin{enumerate}
\item {\it $d_A$ is a metric on $\partial X$ compatible with the cone topology.}
\item {\it The family $r_\xi:\xi\in\partial X$ of geodesic rays based at $x_0$
is conical (with the parameter space $\Lambda=(\partial X,d_A)$)
for the constant $C=A$.}
\end{enumerate}
\end{prop}

\medskip\noindent
{\bf Remarks.}
(1) The question about existence of natural metrics on CAT(0) boundaries was posed
by K.\ Ruane in \cite{Kap}, as Problem 45. The metrics as in the above proposition were then
noticed by the first author (see \cite[Remark 17]{Kap}). Viewed as the family
with any parameters $A>0$ and $x_0\in X$, they provide a tool for doing analysis on CAT(0) boundaries,
since they all are in the same quasi-conformal class. We do not
use this fact in the present paper and thus we skip a rather
straightforward proof of it.

(2) In fact, the family $r_\xi:\xi\in\partial X$ from part (2)
is conical (with respect to the metric $d_A$)
for any constant $C>0$, but we do not need this stronger observation.

\begin{proof} (of Proposition~\ref{d6})
We first show that $d_A$ is a metric. The only thing to be checked is
the triangle inequality. Let $\xi,\eta,\zeta$ be three distinct points of $\partial X$.
Suppose that $d_A(\xi,\eta)=t^{-1}$ and $d_A(\eta,\zeta)=s^{-1}$
and that $s\leqslant t$.
This means that $d_X(r_\xi(t),r_\eta(t))=A=d_X(r_\eta(s),r_\zeta(s))$.
By convexity of the function $u\to d_X(r_\xi(u),r_\eta(u))$ we have
$d_X(r_\xi(s),r_\eta(s))\leqslant {\frac{s}{t}}\cdot A$, and hence
\begin{align*}
d_X(r_\xi(s),r_\zeta(s))\leqslant A+{\frac{s}{t}}\cdot A=\frac{s+t}{t}\cdot A.
\end{align*}
Similarly, by convexity of the function $u\to d_X(r_\xi(u),r_\zeta(u))$, we get
\begin{align*}
d_X\left(r_\xi\left(\frac{st}{s+t}\right),r_\zeta\left(\frac{st}{s+t}
\right)\right)\leqslant A,
\end{align*}
which implies that $d_X(r_\xi(u),r_\zeta(u))=A$ for some $u\geqslant\frac{st}{s+t}$.
Consequently
\begin{align*}
d_A(\xi,\zeta) \leqslant
\frac{s+t}{st}=\frac{1}{s}+\frac{1}{t}=d_A(\xi,\eta)+d_A(\eta, \zeta).
\end{align*}

We now turn to showing that the metric $d_A$ is compatible with the cone topology on $\partial X$.
Recall that the cone topology on $\partial X$ is induced by the basis of open sets of the form
\begin{align*}
U(\xi,u,\epsilon)=\{ \eta\in\partial X \,|\, d_X(r_\xi(u),r_\eta(u))<\epsilon \},
\end{align*}
see \cite[Chapter II, Exercise 8.7]{BrHa}. Observe that any open ball of radius $\rho$ centered at $\eta$
in  $(\partial X,d_A)$ coincides with the base set
$U(\xi,\frac{1}{\rho},A)$.
Conversely, for any set $U(\xi,u,\epsilon)$ and any point $\eta\in U(\xi,u,\epsilon)$
put $\delta:=d_X(r_\xi(u),r_\eta(u))$ and note that $\delta<\epsilon$.
If $A\leqslant\epsilon-\delta$ then the ball in $(\partial X,d_A)$ of radius $1/u$ centered at $\eta$
is clearly contained in $U(\xi,u,\epsilon)$. If $A>\epsilon-\delta$ then one verifies
(again using convexity of an appropriate distance function)
that the ball of radius $\frac{\epsilon-\delta}{Au}$ centered at $\eta$
is contained in $U(\xi,u,\epsilon)$.
Hence the compatibility.

It remains to show that the family of rays $r_\xi:\xi\in\partial X$ is conical
with respect to the metric $d_A$, i.e.\ to verify conditions (1) and (2)
of Definition \ref{d3}. If we choose $C=A$ then, by the definition of $d_A$
and by the monotonicity of functions of the form $u\to d_X(r_\xi(u),r_\eta(u))$,
we can take $\varepsilon=1/N$ in the condition (1) and
$N'=\frac{L}{C}\cdot\frac{1}{\varepsilon'}$
in the condition (2). This finishes the proof.
\end{proof}

\begin{proof}
(of Theorem \ref{d1}; assuming Proposition \ref{d4})
According to Lemma \ref{d5} or Proposition \ref{d6} (depending on the considered case),
there is a conical family of geodesic rays $r_\xi:\xi\in\partial X$ in $X$.
If, contrary to the assertion, $\partial X$ contains a $2$--disk
then, restricting the family of rays to this disk, we get that $X$ is
$D^2$--conical. By Proposition~\ref{d4} the space $X$ is then not AHA,
contradicting the assumption. Hence the theorem.
\end{proof}

\begin{proof}(of Proposition \ref{d4})
We start with two general observations. Let $(\Lambda,\partial\Lambda)$
be a pair of a metric space and its subspace (with restricted metric)
homeomorphic to the pair $(D^2,S^1)$, where $S^1=\partial D^2$.
The following result is classical, and it follows from the Vietoris
approach to homology of metric spaces, as recalled in Subsection~\ref{vietpr},
specifically from Theorem~\ref{v6} and Fact~\ref{v7},
in view of the fact that $H_1(S^1;\mathbb Z)=\mathbb Z\ne 0$.

\medskip\noindent
{\bf Claim 1.} {\it There is $a>0$ such that for any $0<\varepsilon<a$ there is a simplicial
map $f_\epsilon:\Sigma\to P_\varepsilon(\partial\Lambda)$, where $\Sigma$ is a triangulation of $S^1$,
which induces a simplicial cycle homologically nontrivial in $H_1(P_a(\partial\Lambda);\mathbb Z)$.}

\medskip
On the other hand, the following is an elementary observation.

\medskip\noindent
{\bf Claim 2.} {\it For each $\varepsilon>0$, any simplicial map $f:\Sigma\to P_\varepsilon(\partial\Lambda)$,
where $\Sigma$ is any triangulation of $S^1$, extends to a simplicial map
$F:\Delta\to P_\varepsilon(\Lambda)$, where $\Delta$ is a triangulation
of $D^2$ such that $\partial\Delta=\Sigma$.}

\medskip
Now, let $X$ be any $D^2$--conical metric space. We need to show that $X$ is not AHA.
To do this, for some constant $C_0>0$ and for each $R\geqslant C_0$ we construct a simplicial map $h_R:S\to P_{C_0}(X)$,
where $S$ is a triangulation of the $2$--sphere, which has no extension $H:B\to P_R(h_R(S^{(0)}))$,
where $S^{(0)}$ is the $0$--skeleton of $S$ and $B$ is a triangulation of the $3$--ball
such that $\partial B=S$. Existence of such a family of maps clearly contradicts AHA.
The idea of the construction is simple: we obtain such cycles as boundary surfaces of ``truncated cones''
inscribed in a $D^2$--conical family of rays in $X$. Execution of this idea is however
a bit long and technical. We split it into three parts.

\medskip\noindent
{\it Part 1: projections of $\Lambda$ on the spheres in $X$.}

Let $\Lambda$ be a metric space homeomorphic to $D^2$ as in the definition of $D^2$--conicality
(applied to $X$), and denote by $\partial\Lambda$ its subspace corresponding to $S^1$.
For each $t>0$ let $p_t:\Lambda\to X$ be defined by $p_t(\xi)=r_\xi(t)$.
Note that $p_t$ maps elements of $\Lambda$ to the sphere $S_X(x_0,t)$ in $X$,
centered at $x_0$ (the basepoint of all geodesic rays $r_\xi$) and of radius $t$.
We will call these maps \emph{projections} of $\Lambda$ on the corresponding spheres.
In the next claim we list some properties of these projections which
follow directly from the definition.

\medskip\noindent
{\bf Claim 3.}
(1) {\it For any $\xi\in\Lambda$ and any $0\leqslant s<t$ we have
$d_X(p_t(\xi),p_s(\xi))=t-s$}.

(2) {\it For any $\xi,\eta\in\Lambda$  and any $0\leqslant s<t$ we have
$t-s\leqslant d_X(p_t(\xi),p_s(\eta))\leqslant t-s+d_X(p_s(\xi),p_s(\eta))$.}
\medskip

We will frequently use the above two properties of the projections,
sometimes even without explicitly referring to them.

Next claim is our first derived consequence of $D^2$--conicality. In its statement,
$C$ denotes the constant from the definition of $D^2$--conicality, $a$  is a constant
from Claim 1 (for the metric space $\Lambda$ from definition of $D^2$--conicality),
while $\Sigma$
and $\Delta$ run through triangulations of $S^1$ and $D^2$ such that $\partial\Delta=\Sigma$.
Homology appearing in the statement is the simplicial homology with integer coefficients.

\medskip\noindent
{\bf Claim 4.} {\it $\forall{L\geqslant C} \,\, \exists{t_0>0} \,\, \forall{t\geqslant t_0} \,\, \exists{\varepsilon_t>0}
\,\, \forall t_0\leqslant s\leqslant t$
if $f:\Sigma\to P_{\varepsilon_t}(\partial\Lambda)$ induces a homologically nontrivial
cycle in $P_a(\partial\Lambda)$ and if $F:\Delta\to P_{\varepsilon_t}(\Lambda)$
is an extension of $f$, then}
\begin{enumerate}
\item {\it $p_s\circ F:\Delta\to P_C(X)$ is well defined and}
\item {\it $p_s\circ f:\Sigma\to P_C(X)$ induces a homologically nontrivial cycle in
$P_L(p_s(\partial\Lambda))$.}
\end{enumerate}

\medskip\noindent
\emph{Proof of Claim 4:}
First, using condition (2) of Definition \ref{d3}, choose $t_0>0$ such that
$\forall s\geqslant t_0$ if $d_X(p_s(\xi),p_s(\eta))\leqslant L$ then $d_\Lambda(\xi,\eta)\leqslant a$.
Next, using condition (1) of Definition \ref{d3}, choose $\varepsilon_t$ so that if
$d_\Lambda(\xi,\eta)\leqslant\varepsilon_t$ then $d_X(p_s(\xi),p_s(\eta))\leqslant C$ for $s\leqslant t$.
Clearly, $p_s\circ F:\Delta\to P_C(X)$ is then well defined for $s\leqslant t$, which proves assertion (1).

To prove assertion (2), suppose that the cycle $p_s\circ f$ is null-homologous
in $P_L(p_s(\partial\Lambda))$. Let $w$ be a $2$--chain in $P_L(p_s(\partial\Lambda))$
with $\partial w=p_s\circ f$. Lifting the vertices of $w$ consistently to $\partial\Lambda$
(by means of any partial inverse of $p_s$)
we get a $2$--chain $\tilde w$ in $P_a(\partial\Lambda)$ with
$\partial\tilde w=f$. But this contradicts the assumption on $f$,
which finishes the proof of Claim 4.

\medskip\noindent
{\it Part 2: construction of a family of large cycles in $X$.}

Now, using Claim 4, we construct a family of large spherical cycles
$h_R:S\to P_{C_0}(X)$, as mentioned in the paragraph just after Claim 2.
Put $C_0=C+1$, where $C$ is the constant as in Claim 4 (i.e.\ the constant as in the definition
of $D^2$--conicality for $X$).
For any integer $R\geqslant C_0$, put $L=3R$ and let $t_0$ be the constant from Claim 4.
Without loss of generality, we assume that $t_0$ is integer
(this will be useful in the next paragraph).
Put $t=t_0+4R$ and let $\varepsilon_t$ be again as asserted in Claim 4.
Moreover, assume that $\varepsilon_t<a$ (where $a$ is as in Claim 1),
again not losing generality.
Finally, let $f_{\varepsilon_t}:\Sigma\to P_{\varepsilon_t}(\partial\Lambda)$ be a simplicial map
as prescribed by Claim 1, and let $F:\Delta\to P_{\varepsilon_t}(\Lambda)$
be its extension guaranteed by Claim 2.

View the interval $[t_0,t]$ as a $1$--dimensional simplicial complex whose vertices
are all the $4R+1$ integers contained in this interval. Triangulate $\Delta\times[t_0,t]$
in a standard way as the product of simplicial complexes. Then the vertices of $\Delta\times[t_0,t]$
all have the form $(v,k)$, where $v$ is a vertex of $\Delta$ and $k\in[t_0,t]$ is an integer.

Map the vertices of $\Delta\times[t_0,t]$ to $X$ by $(v,k)\to p_k(f_{\varepsilon_t}(v))$
and note that, by the choices of $t_0$ and $\varepsilon_t$ accordingly with
assertion (1) of Claim 4, and by earlier mentioned properties of the maps $p_k$,
adjacent vertices are mapped to points at the distance at most $C+1=C_0$.
Consequently, this map induces the simplicial map $h:\Delta\times[t_0,t]\to P_{C_0}(X)$.
Denote by $S$ the boundary simplicial complex
$\partial(\Delta\times[t_0,t])=\Delta\times\{ t_0,t \}\cup\partial\Delta\times[t_0,t]$,
which is topologically a $2$--sphere. Put $h_R:S\to P_{C_0}(X)$
to be the restriction of $h$ to $S$.

\medskip\noindent
{\it Part 3: large cycles are not contractible along themselves.}

To finish the proof of Proposition \ref{d4}, we need to show that there is no
extension $H:B\to P_R(h_R(S^{(0)}))$, where $S^{(0)}$ is the $0$--skeleton of $S$,
and $B$ is a triangulation of the $3$--ball such that $\partial B=S$.
On the contrary, suppose that such an extension $H$ exists. We will show that
this contradicts assertion (2) of Claim 4. Our argument will be based on
the idea related to the construction of the connecting homomorphism
in the homology exact sequence of a pair.

Denote by $B_0$ the subcomplex of $B$ equal to the union of all $3$--simplices of $B$
whose all vertices are mapped through $H$ to the points of $X$ at distance
at most $t_0+2R$ from the point $x_0$ (the origin of the geodesic rays
in the conical family $\{ r_\lambda \}$).
Consider the boundary $\partial B_0$, which we view as $2$--dimensional simplicial cycle.
Note that $\partial B_0$ contains half of $S$, namely
$S_0=\Delta\times\{ t_0 \}\cup\partial\Delta\times[t_0,t_0+2R]$.
Let $U$ be the remaining part of $\partial B_0$, i.e.\ $\partial B_0=S_0+U$.
Then $U$ is a simplicial $2$--chain and
$\partial U=\partial\Delta\times\{ t_0+2R \}=\Sigma\times\{ t_0+2R \}$.

Now, let $u$ be the chain in the Rips' complex $P_R(h_R(S^{(0)}))$
induced by the simplicial map $H$ restricted to $U$. Clearly, the boundary cycle
$\partial u$ coincides with the cycle induced by $H$ restricted to $\partial U$.
Since $H|_{\partial U}=h_R|_{\Sigma\times\{ t_0+2r \}}$, by the definition of $h_R$
we see that $\partial u$ coincides with the cycle induced by the map
$p_{t_0+2R}\circ f:\Sigma\to P_C(p_{t_0+2R}\circ f(\Sigma^{(0)}))$.

We now make some further observations concerning the chains $U$ and $u$.

\medskip\noindent
{\bf Claim 5.} {\it For any vertex $b$ of $U$ we have
$t_0+R\leqslant d_X(x_0,H(b))\leqslant t_0+2R$.}

\medskip\noindent
\emph{Proof of Claim 5:} Note that each $2$--simplex $\sigma$ of $U$ is a face of
a $3$--simplex of $B$ whose vertex $z$ opposite to $\sigma$ satisfies
$d_X(x_0,H(z))>t_0+2R$. In particular, $b$ is adjacent in $B$ to a vertex $z$
of this form. Since adjacency of $b$ and $z$ in $B$ implies
$d_X(H(b),H(z))\leqslant R$, the claim follows.

\medskip
Recall that $H(U^{(0)})$ is the set of vertices of the chain $u$.
Since $H(U^{(0)})\subseteq H(B^{(0)})\subseteq h_R(S^{(0)})$, it follows from Claim 5,
and from the contemplation of the cylindrical shape of $h_R(S^{(0)})$,
that $H(U^{(0)})\subseteq h_R(\Sigma^{(0)}\times[t_0+R,t_0+2R]^{(0)})$.
This allows to modify the chain $u$ slightly, by shifting its vertices
not contained in $h_R(\Sigma^{(0)}\times\{ t_0+2R \})$ into this set.
More precisely, for any vertex $b\in U^{(0)}$ we have
$H(b)=h_R(v,k)$ for some unique integer $k\in[t_0+R,t_0+2R]$ and some
(not necessarily unique) $v\in\Sigma_0$. Put $H'(b)=h_R(v,t_0+2R)$
if $k<t_0+2R$, and $H'(b)=H(b)$ otherwise.

Observe that, by the definition of $h_R$, we have $d_X(H'(b),H(b))\leqslant R$
for all $b\in U^{(0)}$. Since for adjacent vertices $b,b'\in U^{(0)}$
we have $d_X(H(b),H(b'))\leqslant R$, it follows that $d_X(H'(b),H'(b'))\leqslant 3R=L$
for such vertices. Consequently, $H'$ induces a simplicial map
$U\to P_L(h_R(\Sigma^{(0)}\times\{ t_0+2R \}))$, and this map defines
a simplicial $2$--chain $u'$ in the latter Rips' complex.

By noting that $h_R(\Sigma^{(0)}\times\{ t_0+2R \})\subseteq p_{t_0+2R}(\partial\Lambda)$,
we may view $u'$ as a chain in $P_L(p_{t_0+2R}(\partial\Lambda))$.
Moreover, since $H'$ coincides with $H$ on the boundary $\partial U$,
we have $\partial u'=\partial u=p_{t_0+2R}\circ f$.
This shows that $p_{t_0+2R}\circ f$ is null-homologous in $P_L(p_{t_0+2R}(\partial\Lambda))$,
contradicting assertion (2) of Claim 4. This completes the proof.

\end{proof}

\subsection{Gap functions and gap-conicality.}
\label{gap}

In this subsection we extend the notions and results presented above
to parameter spaces $\Lambda$ equipped with the structure weaker than a metric,
namely a gap function.
We assume the reader is familiar with terminology and results
recalled in Subsection \ref{vietpr}.

A \emph{gap function} on a set $X$ is a function $g: X\times X\to[0,\infty)$
which is symmetric (i.e.\ $g(x,y)=g(y,x)$ for all $x,y\in X$) and such that
$g(x,y)=0$ iff $x=y$. Thus, a gap function satisfies the requirements for a metric,
except the triangle inequality.

Note that any gap function $g$ induces a diameter function $\nu_g$, just as a metric does,
by $\nu_g(\sigma):=\max\{ g(x,y):x,y\in\sigma \}$. If $(X,d)$ is a metric space,
we say that a gap function in $X$ is \emph{compatible} with $d$ if the diameter functions
$\nu_g$ and $\nu_d$ are equivalent.

Let $X$ be a compact metrizable topological space. We say that a gap function
$g$ in $X$ is \emph{compatible with the topology} of $X$ if it is compatible with
some (and hence also every) metric compatible with the topology of $X$.
Compatibility of $g$ with the topology of $X$ obviously can be expressed
without referring to any metric on $X$, as follows.

\begin{fact}
\label{g1}
Let $X$ be a compact metrizable topological space. A gap function $g$ in $X$
is compatible with the topology iff the following two conditions hold:
\begin{enumerate}
\item {\it for each finite open covering $\mathcal U$ of $X$ there is $\varepsilon>0$
such that if $g(x,y)<\varepsilon$ then we have $x,y\in U$ for some $U\in{\mathcal U}$;}
\item {\it $\forall\varepsilon>0$ there is a finite open covering $\mathcal U$ of $X$
such that if $x,y\in U\in{\mathcal U}$ then $g(x,y)<\varepsilon$.}
\end{enumerate}
\end{fact}

We now extend the definition of conicality, Definition \ref{d3},
to the case of parameter spaces equipped with gap functions.

\begin{de}[Gap conicality]
\label{g2}
Let $r_\lambda \colon \lambda \in \Lambda$ be a family of geodesic
rays in a metric space $(X,d_X)$, based at one point $x_0\in X$, parametrized by
a set $\Lambda$ equipped with a gap function $g$. This family is \emph{conical}
(with respect to $g$)
if for some constant $C>0$ the following two conditions are satisfied:
\begin{enumerate}
\item $\forall N>0\;\exists\varepsilon>0\; \forall t\leqslant N$
if $g(\lambda,\mu)\leqslant \varepsilon$ then $d_X(r_\lambda(t),r_\mu(t))\leqslant C$,
\item $\forall L\geqslant C\;\forall\varepsilon'>0\;\exists N'>0\;\forall
t\geqslant N'$ if
$d_X(r_\lambda(t),r_\mu(t))\leqslant L$ then $g(\lambda,\mu)\leqslant \varepsilon'$.
\end{enumerate}

\noindent
A metric space is \emph{gap--$D^2$--conical} if it admits a family
of geodesic rays parameterized by a space $\Lambda$ homeomorphic
to $D^2$, which is conical with respect to some gap function $g$
in $\Lambda$ compatible with the topology.
\end{de}

Clearly, gap--$D^2$--conicality generalizes metric $D^2$--conicality.
Moreover, the following extension of Proposition \ref{d4} holds.

\begin{prop}
\label{g3}
If a metric space is gap--$D^2$--conical
then it is not AHA.
\end{prop}

To prove Proposition \ref{g3}, we may use the same arguments as in the proof of Proposition \ref{d4},
provided we establish appropriate analogs of Claims 1 and 2 from that proof.
To do this, note first that Lemma \ref{v4}, Theorem \ref{v6} and Fact \ref{v7} immediately imply
the following.

\begin{cor}
\label{g4}
Let $g$ be a gap function on $S^1$ compatible with the standard topology.
Then $VH_1(S^1,\nu_g)=\mathbb Z\ne0$. In particular, the following variant of Claim 1
(from the proof of Proposition \ref{d4}) holds:

\noindent
$\exists a>0$ $\forall 0<\varepsilon<a$ there exists a simplicial map
$f:\Sigma\to V_{\varepsilon}(S^1,\nu_g)$, where $\Sigma$ is some triangulation of
$S^1$, which induces a simplicial $1$--cycle homologically nontrivial in $V_a(S^1,\nu_g)$.
\end{cor}

An elementary argument, which we omit, gives also the following analog
of Claim 2 from the proof of Proposition \ref{d4}.

\begin{fact}
\label{g5}
Let $(\Lambda,\partial\Lambda)$ be a pair of a topological
space and its subspace homeomorphic to the pair $(D^2,S^1)$. Let $g$ be a gap
function in $\Lambda$ compatible with the topology. Denote by the same symbol $g$
the restricted gap function in $\partial\Lambda$ (which is clearly compatible
with the restricted topology). Then
for each $\varepsilon>0$, any simplicial map $f\colon \Sigma\to V_\varepsilon(\partial\Lambda,\nu_g)$,
where $\Sigma$ is any triangulation of $S^1$, extends to a simplicial map
$F\colon \Delta\to V_\varepsilon(\Lambda,\nu_g)$, where $\Delta$ is a triangulation
of $D^2$ such that $\partial\Delta=\Sigma$.
\end{fact}

\begin{proof}(of Proposition~\ref{g3})
In view of Corollary \ref{g4} and Fact \ref{g5}, the proof goes along the same line of arguments
as the proof of Proposition \ref{d4}.
\end{proof}

\subsection{Gap-conicality in systolic complexes -- proof of Theorem \ref{d2}}
\label{sysgap}

In this section we prove Theorem \ref{d2}, i.e.\ we show that systolic boundary
of a locally finite systolic complex contains no subspace homeomorphic
to the $2$--disk. As we show below, in view of Proposition \ref{g3} and the fact that systolic complexes
are AHA, to prove Theorem \ref{d2} it suffices to show
the following result (Proposition~\ref{s1}).
In the statement of this result we use the notion of
{\it good geodesic rays}, as introduced in \cite[Definition 3.2]{OP}. We recall some basic properties
of good geodesic rays later in this section, and here we only need two things.
First, that good geodesic rays in a systolic complex $X$ are certain geodesic rays
in the $1$--skeleton $X^{(1)}$ of $X$. Second, that the systolic boundary $\partial X$,
as a set, coincides with the set of all good geodesic rays in $X$, based at
some fixed vertex $O$, quotiened by the equivalence relation of lying at finite
Hausdorff distance in $X$ \cite[Definition 3.7]{OP}.

\begin{prop}[Gap on systolic boundaries]
\label{s1}
Let $X$ be a locally finite systolic complex and $\partial X$ its systolic boundary.
Fix a vertex $O$ in $X$, and for each $\xi\in\partial X$ choose a good geodesic ray
$r_\xi$ based at $O$ and representing $\xi$. There is a gap function $g$ in $\partial X$
satisfying the following two conditions:
\begin{enumerate}
\item {\it $g$ is compatible with the topology of $X$;}
\item {\it the family $r_\xi \colon \xi\in\partial X$ of good geodesic rays in $X^{(1)}$
is conical with respect to $g$.}
\end{enumerate}
\end{prop}

\begin{proof} (of Theorem \ref{d2} assuming Proposition \ref{s1})
Suppose a contrario that there is $D^2\subset\partial X$. By Proposition \ref{s1}(2),
the subfamily $r_\xi:\xi\in D^2$
is then conical with respect to the gap function $g$ restricted to $D^2$. Since,
by Proposition \ref{s1}(1), this restricted gap function is compatible with the topology
of $D^2$, we conclude that the geodesic space $X^{(1)}$ is gap--$D^2$--conical.
By Proposition~\ref{g3}, $X^{(1)}$ is then not AHA, a contradiction.
\end{proof}

\medskip
In the remaining part of this section we prove Proposition \ref{s1}.
We start with recalling some further properties of good geodesic rays.
Part (1) of the lemma below is a special case of \cite[Corollary 3.4]{OP},
and part (2) coincides with Lemma 3.8 in the same paper.

\begin{lem}
\label{s2}
There is some universal constant $D>0$ satisfying the following properties.
For any systolic complex $X$ and any two good geodesic rays $r_1,r_2$ in $X^{(1)}$
based at the same vertex $O$ we have
\begin{enumerate}
\item $d_{X^{(1)}}(r_1(s),r_2(s))\leqslant \frac{s}{t}\cdot
d_{X^{(1)}}(r_1(t),r_2(t))+D$
for any integer $s,t$ such that $0<s<t$;
\item $r_1,r_2$ represent the same point in $\partial X$ (i.e.\ they lie at
finite Hausdorff distance) iff $d_{X^{(1)}}(r_1(u),r_2(u))\leqslant D$ for all positive
integers $u$.
\end{enumerate}
\end{lem}

\begin{proof}(of Proposition \ref{s1})
The proof is divided into three parts. We start with constructing some
gap function $g$ in $\partial X$. Then we show that $g$ is compatible with the topology
of $\partial X$. Finally, we prove that the family $r_\xi:\xi\in\partial X$
is conical with respect to $g$.

\medskip\noindent
\emph{Description of a gap function $g$ in $\partial X$.}

Fix some even integer $A\geqslant 2D+2$, where $D$ is the constant from Lemma \ref{s2}.
Let $\xi,\eta\in\partial X$, $\xi\ne\eta$, and let $r,r'$ be any good geodesic rays
based at $O$ representing $\xi$ and $\eta$, respectively (we will use the notation
$r\in\xi$, $r'\in\eta$ to express the latter relationship). By Lemma \ref{s2},
the sequence $(d_{X^{(1)}}(r(n),r'(n)))_{n\in N}$ converges to $\infty$ as $n\to\infty$.
Consequently, the following number is well defined (i.e.\ it is finite):
\begin{align*}
N(\xi,\eta):=\min\{ n_0:\forall n\geqslant n_0 \, \forall r\in\xi,r'\in\eta  \,\,
d_{X^{(1)}}(r(n),r'(n))\geqslant A \}
\end{align*}
(where $n_0$ and $n$ run through non-negative integers).

Put $g(\xi,\eta)=1/N(\xi,\eta)$ if $\xi\ne\eta$ and $g(\xi,\eta)=0$ otherwise,
and note that $g$ clearly satisfies the requirements of the definition of a gap function,
namely it is symmetric and positive.

\medskip\noindent
\emph{Compatibility of $g$ with the topology of $\partial X$.}

To prove compatibility of the above defined gap function $g$ with the topology of
$\partial X$ we refer to the description of this topology in \cite{OP}, and use
the characterization of compatibility given in Fact \ref{g1}.

Recall that in \cite{OP} the topology of $\partial X$ is introduced by means
of a base for a neighborhood system (which consists of sets that are
not necessarily open in the resulting topology).
More precisely, for each $\xi\in\partial X$ we have a family ${\mathcal N}_\xi$
of sets containing $\xi$, called {\it standard neighborhoods} of $\xi$,
and the whole system ${\mathcal N}_\xi:\xi\in\partial X$
satisfies some appropriate axioms. Open sets are described as those
$U\subseteq X$ for which $\forall \xi\in U$ $\exists Q\in{\mathcal N}_\xi$
such that $Q\subseteq U$. Moreover, each $Q\in{\mathcal N}_\xi$ contains some
open neighborhood of the point $\xi$.
Finally, standard neighborhoods $Q\in{\mathcal N}_\xi$ have the form
\begin{align*}
Q=Q(r,N,R)=\{ \eta\in\partial X:\hbox{ for some }r'\in\eta\hbox{ it holds }
d_{X^{(1)}}(r(N),r'(N))\leqslant R  \},
\end{align*}
for any good geodesic ray $r\in\xi$, and any positive integers $N,R$ with $R\geqslant D+1$
(see \cite[Definition 4.1]{OP}).

We now verify condition (2) of Fact \ref{g1}. Given any $\varepsilon>0$, consider
$N$ such that $1/N<\varepsilon$. Denote by $S_N$ the set of all vertices $v$ of $X$
such that $v=r(N)$ for some good geodesic ray $r$ based at $O$.
By local finiteness of $X$, the set $S_N$ is finite. Moreover, if
$v=r(N)=r'(N)$ then the standard neighborhoods $Q(r,N,R)$ and $Q(r',N,R)$ coincide.
Thus we put $Q(v,N,R):=Q(r,N,R)$ for any good geodesic ray $r$ based at $O$
and such that $r(N)=v$.

Consider the family ${\mathcal U}=\{ U_v:v\in S_N \}$ consisting of the interiors $U_v$
of the standard neighborhoods $Q(v,N,A/2)$ (note that, by the earlier assumption on $A$,
$A/2$ is an integer and $A/2\geqslant D+1$).
Clearly, $\mathcal U$ is a finite family of open subsets of $\partial X$.
We will show that it is as required in condition (2) of Fact \ref{g1}, i.e.\ that
\begin{enumerate}[(i)]
\item $\mathcal U$ is a covering of $\partial X$, and
\item if $\xi,\eta\in U_v$ then $g(\xi,\eta)<\varepsilon$.
\end{enumerate}

To see (i), note that for each $\xi\in\partial X$ and any $r\in\xi$
we have $\xi\in U_{r(N)}$ (because $Q(r(N),N,A/2)\in{\mathcal N}_\xi$;
compare also \cite[Lemma 4.8]{OP}). To show (ii), note that if $\xi,\eta\in U_v$
then $\xi,\eta\in Q(v,N,A/2)$, and consequently
\begin{align*}
d_{X^{(1)}} (r(N),r'(N)) \leqslant d_{X^{(1)}} (r(N),v)+d_{X^{(1)}} (v,r'(N))\leqslant A/2+A/2=A
\end{align*}
for some $r\in\xi$, $r'\in\eta$. Thus $N(\xi,\eta)>N$, and hence
$g(\xi,\eta)=1/N(\xi,\eta)<1/N<\varepsilon$, as required.

We turn to verify condition (1) of Fact \ref{g1}.
Let $\mathcal U$ be a finite open cover of $\partial X$. In the proof of \cite[Proposition 5.6]{OP} it is shown that there exists a positive integer $N$ such that
the family of standard neighborhoods $Q(v,N,A):v\in S_N$ is a refinement of $\mathcal U$.
(Observe that in the proof of this property in \cite{OP} the assumption of uniform local finiteness
of $X$, occurring in the statement of \cite[Proposition 5.6]{OP},  is not used.)

Put $\varepsilon=\frac{1}{2N}$, and suppose that $g(\xi,\eta)<\varepsilon$ for some distinct
$\xi,\eta\in\partial X$. It means that $d_{X^{(1)}}(r(N'),r'(N'))=A$
for some $r\in\xi$, $r'\in\eta$, and some integer $N'>1/\varepsilon=2N$.
By the fact that $A\geqslant 2D+2$, and due to Lemma \ref{s2}(1), we get
$d_{X^{(1)}}(r(N),r'(N))<\frac{1}{2}A+D<A$.
Consequently, we have $\xi,\eta\in Q(w,N,A)$ for $w=r(N)$.
Since the family $Q(v,N,A):v\in S_N$ is a refinement of $\mathcal U$,
it follows that $\xi,\eta\in U$ for some $U\in{\mathcal U}$, as required.

This completes the proof of compatibility of $g$ with the topology of $\partial X$.

\medskip\noindent
\emph{Conicality of the family $r_\xi:\xi\in\partial X$ with respect to $g$.}

We need to check conditions (1) and (2) of Definition \ref{g2}.

To check (1), fix $C=A+3D$ and take any $N$. We claim that $\varepsilon=1/N$
is as required. Indeed, let $\lambda,\mu\in\partial X$ with $g(\lambda,\mu)<\varepsilon$.
This means that $N(\lambda,\mu)>N$ and that for some good geodesic rays $r\in\lambda$, $r'\in\mu$
we have
\begin{align*}
d_{X^{(1)}} (r(N(\lambda,\mu)-1),r'(N(\lambda,\mu)-1))<A.
\end{align*}
Since for any $t\leqslant N$ we also have $t\leqslant N(\lambda,\mu)-1$,
it follows from Lemma \ref{s2}(1) that $d_{X^{(1)}} (r(t),r'(t))<A+D$ for such $t$.
Then, by Lemma \ref{s2}(2) we get $d_{X^{(1)}} (r_\lambda(t),r_\mu(t))<A+D+2D=C$,
as required.

We now turn to condition (2) of Definition \ref{g2}. Consider any $L\geqslant C$ and any
$\varepsilon'>0$. We claim that any $N'\geqslant\frac{L}{(A-D)\varepsilon'}$ is as required.
To see this, note that for any $t\geqslant N'$ we have $t\geqslant\frac{L}{(A-D)\varepsilon'}$.
If we assume that $d_{X^{(1)}}(r_\lambda(t),r_\mu(t))\leqslant L$ then, by Lemma \ref{s2}(1),
for any $n<\frac{A-D}{L}N'$ we have $d_{X^{(1)}}(r_\lambda(n),r_\mu(n))<A$,
and hence $N(\lambda,\mu)\geqslant\frac{A-D}{L}N'$. Consequently,
\begin{align*}
g(\lambda,\mu)=\frac{1}{N(\lambda,\mu)}\leqslant \frac{L}{A-D}\cdot \frac{1}{N'} \leqslant
\frac{L}{A-D}\cdot \frac{(A-D)\varepsilon'}{L}=\varepsilon',
\end{align*}
which completes the proof.

\end{proof}


\end{document}